\documentclass[12pt]{article}

\usepackage{amsfonts,amsmath,amsthm,amssymb,amscd}
\usepackage[dvips]{graphicx}

\theoremstyle{definition}
\newtheorem{theorem}{Theorem}

\newtheorem{proposition}[theorem]{Proposition}
\newtheorem{corollary}[theorem]{Corollary}

\numberwithin{equation}{section}
\numberwithin{theorem}{section}

\newcommand{\C}{\mathbb{C}}
\newcommand{\Z}{\mathbb{Z}}

\newcommand{\Res}{\textrm{Res}}

\begin{document}

\begin{center}
{\bf{\Large Rational solutions \\ of the Noumi and Yamada system of type $A_4^{(1)}$}}
\end{center}

\begin{center}
Department of Engineering Science, Niihama National College of Technology,\\
7-1 Yagumo-chou, Niihama, Ehime, 792-8580. 
\end{center} 

\begin{center}
By Kazuhide Matsuda
\end{center}
\vskip 1mm
\par
\quad {\bf Abstract:} \,
We completely classify the rational solutions of the Noumi and Yamada system of type $A_4^{(1)}$, which is a generalization of the fourth Painlev\'e equation.
The rational solutions are classified to three by the B\"acklund transformation
group. 

\section*{Introduction}
Paul Painlev\'e and his coworkers 
\cite{Painleve, Gambier} 
intended to find ``new transcendental functions'' defined by second order nonlinear differential equations. 
For this purpose, 
they investigated which second order ordinary differential equations of the form 
$$
y^{\prime\prime}
=F(t;y,y^{\prime}),
$$
where ${}^{\prime}=d/dt$ and $F$ is rational in $y$ and $y^{\prime}$ and analytic in $t,$ 
have the property that 
the solutions have no movable branch points, i.e., 
the locations of the multi-valued singularities are independent of the particular solution chosen, 
and therefore, dependent only on the equation; 
this is known as the Painlev\'e property. 
As a result, 
the differential equations are either 
integrable in terms of previously known functions 
(such as elliptic functions or are equivalent to linear differential equations) 
or reducible to  
one of the following six equations:
\begin{alignat*}{3}
&P_{\mathrm{I}} &  &:        & \quad  &y^{\prime \prime}=6y^2+t,    \\
&P_{\mathrm{II}} &  &:     &        &y^{\prime \prime}=2y^3+3ty+\alpha,   \\
&P_{\mathrm{III}} &  &:    &  
&y^{\prime \prime}=\frac{1}{y}(y^{\prime})^2-\frac{1}{t}y^{\prime}
+\frac{1}{t}(\alpha y^2+\beta) + \gamma y^3+ \frac{\delta}{y}, \\     
&P_{\mathrm{IV}} &   &:    & 
&y^{\prime \prime}=\frac{1}{2y}(y^{\prime})^2+\frac32 y^3 +4t y^2
+2(t^2-\alpha)y+\frac{\beta}{y},  \\
&P_{\mathrm{V}}  &    &:      & \quad
&y^{\prime \prime}=
\left(\frac{1}{2y}+\frac{1}{y-1}\right)(y^{\prime})^2-\frac{1}{t}y^{\prime}  
+\frac{(y-1)^2}{t^2}\left(\alpha y +\frac{\beta}{y}\right)
+\gamma \frac{y}{t}+ \delta \frac{y(y+1)}{y-1},    
\end{alignat*}
\begin{alignat*}{3}
&P_{\mathrm{VI}}:   &\quad   
&y^{\prime \prime}=
\frac12 \left(\frac{1}{y}+\frac{1}{y-1}+\frac{1}{y-t} \right) (y^{\prime})^2
- \left(\frac{1}{t}+\frac{1}{t-1}+\frac{1}{y-t} \right) y^{\prime}  \\
& & &\hspace{25mm}  +\frac{y(y-1)(y-t)}{t^2(t-1)^2} 
\left( \alpha+\frac{\beta t}{y^2}+\frac{\gamma(t-1)}{(y-1)^2}+\delta \frac{t(t-1)}{(y-t)^2} \right), 
\end{alignat*}
where ${}^{\prime}=d/dt$ and $\alpha, \beta, \gamma, \delta$ are all complex parameters.
\par
The Painlev\'e equations were first derived from the viewpoint of pure mathematics, 
but recently, it was discovered that they are related to a wide range of 
physical problems, such as statistical mechanics and quantum field theory 
(see reference list in \cite{ItsNovokshenov}). 
A remarkable example is that 
the 2-point correlation functions of the rectangular Ising model 
in the scaling limit admit closed expressions 
in terms of the third Painlev\'e transcendents \cite{WuMcCoyTracyBarouch}. 
Furthermore, 
around 1980, 
the Painlev\'e equations were derived as the similarity reductions of 
some partial differential equations in nonlinear wave theory and relativity. 
For example, 
Clarkson \cite{Clarkson} showed that 
the Boussinesq equation has a reduction in the second and fourth Painlev\'e equations, 
and 
rational solutions of the Boussinesq equation are expressed 
in terms of special polynomials associated with 
the rational solutions of the second and fourth Painleve equations. 
Therefore, the study of solutions of the Painlev\'e equations 
is important 
not only from the viewpoint of pure mathematics, 
but also from the viewpoint of physics.
\par
In this paper, 
our concern is with ``classical solutions'' and ``B\"acklund transformations.'' 
The Painlev\'e equations, except for $P_{\mathrm I},$ have 
 ``classical solutions'' which are expressible 
in terms of rational, algebraic or classical special functions for certain values of the parameters, 
and ``B\"acklund transformations'' 
which relate one solution to another solution of the same equation with different values of the parameters. 
\par
The rational solutions of $P_{\mathrm J} \,\,(\mathrm{J}=\mathrm{II, III, IV, V, VI})$ were classified by 
Yablonski and Vorobev \cite{Yab:59,Vorob}, 
Gromak \cite{Gr:83,Gro}, 
Murata \cite{Mura1}, 
Kitaev, Law and McLeod \cite{Kit-Law-McL}, 
Yuang and Li \cite{YuangLi}, 
Mazzocco \cite{Mazzo}. 
\par
Okamoto \cite{oka1, oka2, oka3, oka4} showed that 
the B\"acklund transformation groups are isomorphic 
to the extended affine Weyl groups. 
For $ P_{\mathrm{ II}}, P_{\mathrm{III}}, P_{\mathrm{IV}}, P_{\mathrm{V}},$ and $P_{\mathrm{VI}}$, 
the B\"acklund transformation groups correspond to 
$ 
A^{ ( 1 ) }_1, 
A^{ ( 1 ) }_1 \bigoplus A^{ ( 1 ) }_1, 
A^{ ( 1 ) }_2, 
A^{ ( 1 ) }_3,$ and $ D^{ ( 1 ) }_4 
$, 
respectively.
\par
Noumi and Yamada \cite{NoumiYamada-A} discovered the equation of type $A^{(1)}_l \,(l\geq 2)$, whose B\"acklund 
transformation group is isomorphic to $\tilde{W}(A^{(1)}_l)$. 
This equation is called the Noumi and Yamada system of type $A_l^{(1)}.$ 
The Noumi and Yamada systems of types $A_2^{(1)}$ and $A_3^{(1)}$ correspond to the fourth and fifth Painlev\'e equations, respectively. 
Noted is the fact that 
Murata \cite{Mura1} and Kitaev, Law and McLeod \cite{Kit-Law-McL} classified the rational solutions 
of the fourth and fifth Painlev\'e equations 
using the B\"acklund transformations, respectively. 
In this paper, 
we completely classify the rational solutions of the Noumi and Yamada system of type $A_4^{(1)},$ 
which is defined by
\begin{equation*}
A^{(1)}_4(\alpha_j)_{0\leq j \leq 4}\,\,
\begin{cases}
f_0^{\prime}=f_0(f_1-f_2+f_3-f_4)+\alpha_0 \\ 
f_1^{\prime}=f_1(f_2-f_3+f_4-f_0)+\alpha_1 \\  
f_2^{\prime}=f_2(f_3-f_4+f_0-f_1)+\alpha_2 \\  
f_3^{\prime}=f_3(f_4-f_0+f_1-f_2)+\alpha_3 \\
f_4^{\prime}=f_4(f_0-f_1+f_2-f_3)+\alpha_4 \\
f_0+f_1+f_2+f_3+f_4=t,
\end{cases}
\end{equation*}
where $\alpha_i \,(0\leq i \leq 4)$ are all complex parameters. 
For $A_4^{(1)}(\alpha_j)_{0\leq j \leq 4},$ 
we consider the suffix of $f_i$ and $\alpha_i$ as elements of $\mathbb{Z}/5\mathbb{Z}$. 
Since $\sum_{k=0}^4 f_k=t,$ it follows that 
$$
\alpha_0+\alpha_1+\alpha_2+\alpha_3+\alpha_4=1,
$$
and 
$A_4^{(1)}(\alpha_j)_{0\leq j \leq 4}$ is an essentially nonlinear equation with the fourth order. 
Setting $f_3 \equiv f_4 \equiv 0,$ 
we obtain the Noumi and Yamada system of type $A_2^{(1)},$ 
which is defined by
\begin{equation*}
A^{(1)}_2(\alpha_0,\alpha_1,\alpha_2) \,\,
\begin{cases}
f_0^{\prime}=f_0(f_1-f_2)+\alpha_0 \\ 
f_1^{\prime}=f_1(f_2-f_0)+\alpha_1 \\  
f_2^{\prime}=f_2(f_0-f_1)+\alpha_2 \\  
f_0+f_1+f_2=t,
\end{cases}
\end{equation*}
and is equivalent to the fourth Painlev\'e equation, $P_{\mathrm{IV}}.$ 
$A_4^{(1)}(\alpha_j)_{0\leq j \leq 4}$ is  
the first equation of the Noumi and Yamada systems  of types $A_l^{(1)}\,(l\geq 2)$ which is not the 
original Painlev\'e equations. 
We note that 
Veselov, Shabat \cite{VeselovShabat} and Adler \cite{Adler} 
studied the symmetric forms of the Painlev\'e equations 
from the viewpoint of soliton theory.  
\par
$A_4^{(1)}(\alpha_j)_{0\leq j \leq 4}$ has the B\"acklund transformations, 
$s_0, s_1,s_2, s_3, s_4$ and $\pi$: 
\begin{center}
\begin{tabular}{|c||c|c|c|c|c|c|}
\hline
$x$ & $s_0(x)$ & $s_1(x)$ & $s_2(x)$ & 
$s_3(x)$ & $s_4 (x)$ & $\pi(x)$ \\
\hline
$f_0$ & $f_0$ & $f_0-\alpha_1/f_1$ & $f_0$ & 
$f_0$ & $f_0+\alpha_4/f_4$ & $f_1$ \\
\hline
$f_1$ & $f_1+\alpha_0/f_0$ & $f_1$ & 
$f_1-\alpha_2/f_2$ & $f_1$ & $f_1$ & $f_2$ \\
\hline
$ f_2 $ & $ f_2 $ & $ f_2 + \alpha_1/f_1 $ & 
$ f_2 $ & $ f_2 - \alpha_3/f_3 $ & $ f_2 $ & $ f_3 $ \\
\hline
$ f_3 $ & $ f_3 $ & $ f_3 $ & $ f_3 + \alpha_2/f_2 $ & 
$ f_3 $ & $ f_3 - \alpha_4/f_4 $ & $ f_4 $ \\
\hline
$ f_4 $ & $ f_4 - \alpha_0/f_0 $ & $ f_4 $ & $ f_4 $ & 
$ f_4 + \alpha_3/f_3 $ & $ f_4 $ & $ f_0 $ \\
\hline
\hline
$ \alpha_0 $ & $ - \alpha_0 $ & $ \alpha_0 + \alpha_1 $ & 
$ \alpha_0 $ & $ \alpha_0 $ & $ \alpha_0 + \alpha_4 $ & $ \alpha_1 $ \\
\hline
$ \alpha_1 $ & $ \alpha_1 + \alpha_0 $ & $ - \alpha_1 $ & 
$ \alpha_1 + \alpha_2 $ & $ \alpha_1 $ & $ \alpha_1 $ & $ \alpha_2 $ \\
\hline
$ \alpha_2 $ & $ \alpha_2 $ & $ \alpha_2 + \alpha_1 $ & 
$ - \alpha_2 $ & $ \alpha_2 + \alpha_3 $ & $ \alpha_2 $ & $ \alpha_3 $ \\
\hline
$ \alpha_3 $ & $ \alpha_3 $ & $ \alpha_3 $ & $ \alpha_3 + \alpha_2 $ & 
$ - \alpha_3 $ & $ \alpha_3 + \alpha_4 $ & $ \alpha_4 $ \\
\hline
$ \alpha_4 $ & $ \alpha_4 + \alpha_0 $ & $ \alpha_4 $ & $ \alpha_4 $ & 
$ \alpha_4 + \alpha_3 $ & $ - \alpha_4 $ & $ \alpha_0 $ \\
\hline
\end{tabular}
\end{center}
The B\"acklund transformation group $\langle s_0, s_1, s_2, s_3, s_4, \pi \rangle$
is isomorphic to the extended affine Weyl group $\tilde{W}(A^{(1)}_4)$.
If $f_i \equiv 0$ for some $i=0,1,2,3,4,$ 
which implies that $\alpha_i=0,$ 
we then consider $s_i$ as the identical transformation 
which is given by 
$$
s_i(f_j)=f_j \,\, \mathrm{and} \,\,s_i(\alpha_j)=\alpha_j \,\,(j=0,1,2,3,4).
$$
\par
The result of this paper is that the rational solutions of the Noumi and Yamada system of type $A_4^{(1)}$ are 
classified into three types by the B\"acklund transformations. 
The main theorem of this paper was announced in \cite{Matsu} and 
is as follows: 
\begin{theorem}
\label{thm:a4main}
{\it
For a rational solution of $A^{(1)}_4(\alpha_j)_{0\leq j \leq 4},$ 
by some B\"acklund transformations, 
the solution and parameters can be transformed so that 
one of the following occurs:
\begin{align*}
&\rm{Type \,A} & 
&(f_0 ,f_1, f_2, f_3, f_4)=(t,0,0,0,0), &
&(\alpha_0, \alpha_1, \alpha_2, \alpha_3, \alpha_4) =
(1,0,0,0,0), \\
&\rm{Type \,B}& 
&(f_0 ,f_1, f_2, f_3, f_4) =
(t/3, t/3,  t/3, 0, 0), &  
&(\alpha_0, \alpha_1, \alpha_2, \alpha_3, \alpha_4)= 
(1/3, 1/3, 1/3,0,0),  \\
&\rm{Type \,C} &
&(f_0 ,f_1, f_2, f_3, f_4) =
(t/5, t/5, t/5, t/5, t/5), &  
&(\alpha_0, \alpha_1, \alpha_2, \alpha_3, \alpha_4)= 
(1/5, 1/5, 1/5, 1/5, 1/5). 
\end{align*}
\par
Furthermore, $A^{(1)}_4(\alpha_j)_{0\leq j \leq 4}$ has a rational solution if and only if one of the following occurs:
\newline
{\rm (1)} \quad $\alpha_0, \alpha_1, \alpha_2, \alpha_3, \alpha_4 \in \Z;$
\newline
{\rm (2)} \quad for some $ i = 0, 1, 2,3, 4$,
$$
\quad  
(\alpha_i, \alpha_{i+1}, \alpha_{i+2}, \alpha_{i+3}, \alpha_{i+4}) \equiv
\left\{
\begin{matrix}
\pm 1/3(1, 1, 1, 0, 0) \quad {\rm mod}   \, \Z \\
\pm 1/3(1, -1, -1, 1, 0)\quad {\rm mod}   \, \Z;
\end{matrix}
\right.
$$
\newline
{\rm (3)} \quad for some $i = 0, 1, 2,3, 4$,
$$
(\alpha_i, \alpha_{i+1}, \alpha_{i+2}, \alpha_{i+3}, \alpha_{i+4}) \equiv
\left\{
\begin{matrix}
  j/5(1, 1, 1, 1, 1) \quad {\rm mod} \, \Z    \\
  j/5(1, 2, 1, 3, 3) \quad {\rm mod} \, \Z,   
\end{matrix}
\right.
$$
with some $ j=1, 2, 3, 4. $
\par 
Cases $\mathrm{(1)}, \mathrm{(2)}$ and $\mathrm{(3)}$ correspond to $\mathrm{Type \,\,A},$ $\mathrm{Type \,\,B}$ and $\mathrm{Type \,\,C}$, respectively. 
}
\end{theorem}
\par 
This paper is organized as follows. 
In Section 1, for $A^{(1)}_4(\alpha_i)_{0\leq i \leq 4},$ 
we determine the meromorphic solutions at $t=\infty.$ 
We then find that 
the residues of $f_i \,\,(0\leq i \leq 4)$ at $t=\infty$ are expressed by the parameters 
$\alpha_i \,\,(0\leq i \leq 4).$ 
Moreover, 
we obtain three classes of meromorphic solutions at $t=\infty,$ Type A, Type B and Type C. 
\par
In section 2, following Tahara \cite{Tahara}, 
we investigate the meromorphic solutions at $t=c\in\mathbb{C}$ 
such that some of $(f_i)_{0\leq i \leq 4}$ have a pole at $t=c.$ 
We then see that the residues of $f_i \,(0\leq i \leq 4)$ at $t=c$ are integers. 
Therefore, it follows that $\mathrm{Res}_{t=\infty} f_i \in\mathbb{Z},$ 
which provides the necessary conditions for $A^{(1)}_4(\alpha_i)_{0\leq i \leq 4}$ to have rational solutions. 
Especially, the necessary conditions are expressed by the parameters.  
\par
In Section 3, following Noumi and Yamada \cite{NoumiYamada-B}, 
we first introduce the Hamiltonian $H$ of $A_4^{(1)}(\alpha_i)_{0 \leq i \leq 4}$ 
and its principal part $\hat{H}.$
We next calculate the residues of $\hat{H}$ at $t=\infty, c,$ 
which are given by the parameters, $\alpha_j \,(0\leq j \leq 4).$ 
We last show the relationship between the rational solutions and $\hat{H}.$ 
\par
In Section 4, 
following Noumi and Yamada \cite{NoumiYamada-B}, 
we first introduce the shift operators. 
We next prove that 
for a rational solution of $A^{(1)}_4(\alpha_i)_{0\leq i \leq 4},$ 
by some B\"acklund transformations, 
the parameters can be transformed so that $\alpha_i \in\mathbb{Q}$ and $0\leq \alpha_i \leq 1 \,\,(0\leq i \leq 4).$ 
\par
In Section 5, we obtain the necessary conditions for $A^{(1)}_4(\alpha_i)_{0\leq i \leq 4}$ 
to have rational solutions of Type A, Type B and Type C. 
\par
In Section 6, 
we determine the rational solutions of $A_4^{(1)}(\alpha_i)_{0\leq i \leq 4}$ such that 
the parameters satisfy $0\leq \alpha_i \leq 1\,\,(0\leq i \leq 4).$ 
In Section 7, 
we prove the main theorem.

\, {\it Acknowledgments.} \,
The author wishes to express his sincere thanks to Professor Yousuke Ohyama.

\section{Meromorphic Solutions at $t=\infty$}
In this section, 
for $A_4^{(1)}(\alpha_j)_{0\leq j \leq 4},$ 
we determine 
the meromorphic solutions at $t=\infty,$
and see that 
the residues of $f_j \,\,(0\leq j \leq 4)$ at $t=\infty$ are expressed by the parameters 
$\alpha_i \,\,(0\leq i \leq 4).$ 
\par
For this purpose, 
we set  
\begin{equation*}
\begin{cases}
f_0 =a_{\infty, n_0}t^{n_0}+a_{\infty, n_0-1}t^{n_0-1}+\cdots+a_{\infty,0}+a_{\infty,-1}t^{-1}+\cdots,   \\
f_1 =b_{\infty, n_1}t^{n_1}+b_{\infty, n_1-1}t^{n_1-1}+\cdots+b_{\infty,0}+b_{\infty,-1}t^{-1}+\cdots,   \\
f_2 =c_{\infty, n_2}t^{n_2}+c_{\infty, n_2-1}t^{n_2-1}+\cdots+c_{\infty,0}+c_{\infty,-1}t^{-1}+\cdots,   \\
f_3 =d_{\infty, n_3}t^{n_3}+d_{\infty, n_3-1}t^{n_3-1}+\cdots+d_{\infty,0}+d_{\infty,-1}t^{-1}+\cdots,   \\
f_4 =e_{\infty, n_4}t^{n_4}+e_{\infty, n_4-1}t^{n_4-1}+\cdots+e_{\infty,0}+e_{\infty,-1}t^{-1}+\cdots,   \\
f_5 =f_{\infty, n_5}t^{n_5}+f_{\infty, n_5-1}t^{n_5-1}+\cdots+f_{\infty,0}+f_{\infty,-1}t^{-1}+\cdots,   \\
\end{cases}
\end{equation*}
where $n_0,n_1,n_2,n_3,n_4$ are all integers.

\subsection{The case where  $f_i \,(0\leq i \leq 4)$ are all holomorphic at $t=\infty$}

\begin{proposition}
\label{prop:t=inf-holo}
{
For $A_4^{(1)}(\alpha_j)_{0\leq j \leq 4},$ 
there exists no solution such that 
$f_i \,(0\leq i \leq 4)$ are all holomorphic at $t=\infty.$ 
}
\end{proposition}

\begin{proof}
Considering $f_0+f_1+f_2+f_3+f_4=t, $ 
we can prove the proposition.

\end{proof}

\subsection{The case where one of $f_i \,(0\leq i \leq 4)$ has a pole at $t=\infty$}

\begin{proposition}
\label{prop:t=inf-one-(1)}
{\it 
Suppose that 
for $A^{(1)}_4(\alpha_j)_{0\leq j \leq 4},$ 
there exists a solution such that 
for some $i=0,1,2,3,4,$ 
$f_i$ has a pole at $t=\infty$ and $f_{i+1},f_{i+2},f_{i+3},f_{i+4}$ are all holomorphic at $t=\infty.$ 
Then, 
\begin{equation*}
\begin{cases}
f_i 
&= 
t 
+ 
( 
-\alpha_{i+1}+\alpha_{i+2}-\alpha_{i+3} 
+\alpha_{i+4} 
) 
t^{-1} 
+ 
\cdots, \\
f_{i+1} 
&= 
\alpha_{i+1} t^{-1} + \cdots, \\
f_{i+2} 
&= 
-\alpha_{i+2} 
t^{-1} + \cdots, \\
f_{i+3} 
&= 
\alpha_{i+3} 
t^{-1} + \cdots, \\
f_{i+4} 
&= 
- 
\alpha_{i+4}  
t^{-1} + \cdots. 
\end{cases}
\end{equation*}
}
\end{proposition}

\begin{proof}
By $\pi$, we assume that $f_0$ has a pole at $t=\infty$. 
Since $\sum_{k=0}^4 f_k=t,$ it follows that 
\begin{equation*}
n_0=1, n_j \leq 0 \,\,(1 \leq j \leq 4) \,\, \text{and} \,\,a_{\infty,1}=1.
\end{equation*}
By comparing the coefficients of the term $t, t^0$ in 
\begin{equation*}
f_1^{\prime}
=f_1(f_2-f_3+f_4-f_0) + \alpha_1,
\end{equation*} 
we then get 
\begin{equation*}
b_{\infty,0}=0,\,\, b_{\infty,-1}=\alpha_1.
\end{equation*}
In the same way, we obtain
\begin{equation*}
c_{\infty,0}=0,    c_{\infty,-1}=-\alpha_2,     \,\,\,
d_{\infty,0}=0,    d_{\infty,-1}=\alpha_3,  \,\,\,
e_{\infty,0}=0,    e_{\infty,-1}=-\alpha_4. 
\end{equation*}
Since $\sum_{k=0}^4 f_k=t,$ it follows that 
\begin{equation*}
a_{\infty,0}=0, a_{\infty,-1}=-\alpha_1+\alpha_2-\alpha_3+\alpha_4.
\end{equation*}

\end{proof}

\begin{proposition}
\label{prop:t=inf-one-(2)}
{\it 
Suppose that 
for $A^{(1)}_4(\alpha_j)_{0\leq j \leq 4},$ 
there exists a solution such that 
for some $i=0,1,2,3,4,$ 
$f_i$ has a pole at $t=\infty$ and $f_{i+1},f_{i+2},f_{i+3},f_{i+4}$ are all holomorphic at $t=\infty.$ 
It is then unique. 
}
\end{proposition}

\begin{proof}
By $\pi$ and Proposition \ref{prop:t=inf-one-(1)}, 
we can set 
\begin{equation*}
\begin{cases}
f_0&=t+ a_{\infty,-1}t^{-1}+\sum^{k=\infty}_{k=2 }a_{\infty,-k}t^{-k}, \\
f_1&=   b_{\infty,-1}t^{-1}+\sum^{k=\infty}_{k=2}b_{\infty,-k}t^{-k},  \\
f_2&=   c_{\infty,-1}t^{-1}+\sum^{k=\infty}_{k=2}c_{\infty,-k}t^{-k},  \\
f_3&=   d_{\infty,-1}t^{-1}+\sum^{k=\infty}_{k=2 }d_{\infty,-k}t^{-k},  \\
f_4&=   e_{\infty,-1}t^{-1}+\sum^{k=\infty}_{k=2 }e_{\infty,-k}t^{-k}, 
\end{cases}
\end{equation*}
where $a_{\infty,-1}, b_{\infty,-1}, c_{\infty,-1}, d_{\infty,-1}, e_{\infty,-1}$ have been determined in Proposition \ref{prop:t=inf-one-(1)} and 
\begin{equation*}
\begin{cases}
a_{\infty,-1}=-\alpha_1+\alpha_2-\alpha_3+\alpha_4  \\
b_{\infty, -1}=\alpha_1, \, c_{\infty, -1}=-\alpha_2, \, d_{\infty, -1}=\alpha_3, e_{\infty, -1}=-\alpha_4.
\end{cases}
\end{equation*}
By comparing the coefficients of the terms $t^{-k} \,\,(k \geq 1)$ in
\begin{equation*}
\begin{cases}
f_1^{\prime}=f_1(f_2-f_3+f_4-f_0)+\alpha_1 \\  
f_2^{\prime}=f_2(f_3-f_4+f_0-f_1)+\alpha_2 \\  
f_3^{\prime}=f_3(f_4-f_0+f_1-f_2)+\alpha_3 \\
f_4^{\prime}=f_4(f_0-f_1+f_2-f_3)+\alpha_4,  
\end{cases}
\end{equation*}
we then get
\begin{equation*}
\begin{cases}
b_{\infty,-(k+1)}=b_{\infty,-k}(k-1)
+\sum b_{\infty,-l}(c_{\infty,-m}-d_{\infty,-m}+e_{\infty,-m}-a_{\infty,-m}) \\
c_{\infty,-(k+1)}=-c_{\infty,-k}(k-1)
-\sum  c_{\infty, -l}(d_{\infty, -m}-e_{\infty, -m}+a_{\infty, -m}-b_{\infty, -m}) \\
d_{\infty, -(k+1)}=d_{\infty, -k}(k-1)
+\sum d_{\infty,-l }(e_{\infty, -m}-a_{\infty, -m}+b_{\infty, -m}-c_{\infty, -m}) \\
e_{\infty,-(k+1)}=-e_{\infty, -k}(k-1)
-\sum e_{\infty, -l}(a_{\infty, -m}-e_{\infty, -m}+c_{\infty, -m}-d_{\infty, -m}),
\end{cases}
\end{equation*}
where the sums extend over the positive integers $l$ and $m$ such that $l+m=k.$ 
Since $\sum_{j=0}^4 f_j=t,$ 
it follows that 
\begin{equation*}
a_{\infty, -(k+1)}
=
-b_{\infty, -(k+1)}-c_{\infty, -(k+1)}-d_{\infty, -(k+1)}-e_{\infty, -(k+1)}.
\end{equation*}
Therefore, 
the coefficients 
$a_{\infty,-k},b_{\infty,-k},c_{\infty, -k},d_{\infty, -k},e_{\infty, -k} \,\,(k \geq 1)$ 
are inductively determined.

\end{proof}

From the proof of Proposition \ref{prop:t=inf-one-(2)}, 
we can prove the following corollary.

\begin{corollary}
\label{coro:t=inf-one-identically-zero}
{\it
Suppose that 
for $A^{(1)}_4(\alpha_j)_{0\leq j \leq 4},$ 
there exists a solution such that 
for some $i=0,1,2,3,4,$ 
$f_i$ has a pole at $t=\infty$ and $f_{i+1},f_{i+2},f_{i+3},f_{i+4}$ are all holomorphic at $t=\infty.$ 
Then, 
\begin{equation*}
\begin{cases}
f_{i+1} \equiv 0 \,\,  \text{ if and only if} \,\,\alpha_{i+1}=0, \\
f_{i+2} \equiv 0 \,\,  \text{ if and only if} \,\,\alpha_{i+2}=0, \\
f_{i+3} \equiv 0 \,\,  \text{ if and only if} \,\,\alpha_{i+3}=0, \\
f_{i+4} \equiv 0 \,\,  \text{ if and only if} \,\,\alpha_{i+4}=0. 
\end{cases}
\end{equation*}
}
\end{corollary}

\subsection{The case where two of $f_i\,\,(0\leq i \leq 4)$ have a pole at $t=\infty$}
By $\pi,$ 
we have only to consider the following two cases:
\newline
(1) \quad for some $i=0,1,2,3,4,$ $f_i, f_{i+1}$ both have a pole at $t=\infty,$ \\
(2) \quad for some $i=0,1,2,3,4,$ $f_i, f_{i+2}$ both have a pole at $t=\infty.$

\subsubsection{The case where $f_i, f_{i+1}$ have a pole at $t=\infty$}

\begin{proposition}
\label{prop:t=inf-(f0,f1)-(1)}
{\it 
For $A^{(1)}_4(\alpha_j)_{0\leq j \leq 4},$ 
there exists no solution such that 
for some $i=0,1,2,3,4,$ 
$f_i, f_{i+1}$ both have a pole at $t=\infty$ and $f_{i+2},f_{i+3},f_{i+4}$ are all holomorphic at $t=\infty.$ 
}
\end{proposition}

\begin{proof}
Suppose that 
$A^{(1)}_4(\alpha_j)_{0\leq j \leq 4}$ has such a solution. 
By $\pi,$ 
we can then assume that 
$f_0, f_{1}$ both have a pole at $t=\infty$ and $f_{2},f_{3},f_{4}$ are all holomorphic at $t=\infty.$ 
Thus, 
we note that $n_0,n_1\geq 1,$ $a_{\infty, n_0}b_{\infty, n_1}\neq 0$ and $n_2,n_3,n_4\leq 0.$ 
\par
By comparing the coefficients of the term $t^{n_0+n_1}$ in 
$$
f_1^{\prime}=f_1(f_2-f_3+f_4-f_0)+\alpha_1,
$$
we have $0=-b_{\infty,n_1}a_{\infty,n_0},$ 
which is contradiction.

\end{proof}

\subsubsection{The case where $f_i, f_{i+2}$ have a pole at $t=\infty$}

\begin{proposition}
\label{prop:t=inf-(f0,f2)-(1)}
{\it 
For $A^{(1)}_4(\alpha_j)_{0\leq j \leq 4},$ 
there exists no solution such that 
for some $i=0,1,2,3,4,$ 
$f_i, f_{i+2}$ both have a pole at $t=\infty$ and $f_{i+1},f_{i+3},f_{i+4}$ are all holomorphic at $t=\infty.$ 
}
\end{proposition}

\begin{proof}
The proposition can be proved in the same way as Proposition \ref{prop:t=inf-(f0,f1)-(1)}.

\end{proof}

\subsection{The case where three of $f_i\,\,(0\leq i \leq 4)$ have a pole at $t=\infty$}
By $\pi,$ we have only to consider the following two cases: 
\newline
(1) \quad for some $i=0,1,2,3,4,$ $f_i,f_{i+1},f_{i+2}$ all have a pole at $t=\infty,$ \\
(2) \quad for some $i=0,1,2,3,4,$ $f_i,f_{i+1},f_{i+3}$ all have a pole at $t=\infty$.

\subsubsection{The case where $f_i,f_{i+1},f_{i+2}$ have a pole at $t=\infty$}

\begin{proposition}
\label{prop:t=inf-f0,f1,f2-(1)}
{\it
Suppose that 
for $A^{(1)}_4(\alpha_j)_{0\leq j \leq 4},$ 
there exists a solution such that 
for some $i=0,1,2,3,4,$ 
$f_i,f_{i+1},f_{i+2}$ all have a pole at $t=\infty$ and $f_{i+3},f_{i+4}$ are both holomorphic at $t=\infty.$ 
Then, $n_{i}=n_{i+1}=n_{i+2}=1.$
}
\end{proposition}

\begin{proof}
By $\pi,$ 
we can assume that $f_0,f_1,f_2$ all have a pole at $t=\infty$ and $f_3,f_4$ are both holomorphic $t=\infty.$ 
It then follows that 
$n_0, n_1, n_2\geq 1,$ $a_{\infty,n_0}b_{\infty,n_1}c_{\infty,n_2}\neq 0$ and $n_3, n_4\leq 0.$ 
Since $\sum_{k=0}^4 f_k=t,$ 
one of the following four cases occurs:
\begin{align*}
&\mathrm{(i)} \quad n_0=n_1 > n_2 \geq 1 & 
&\mathrm{(ii)} \quad n_1=n_2 >n_0 \geq 1  \\
&\mathrm{(iii)} \quad n_2=n_0 >n_1 \geq 1 & 
&\mathrm{(iv)} \quad n_0=n_1=n_2 \geq 1.  
\end{align*}
\par
We first treat case (i).  
By comparing the coefficients of the term $t^{n_0+n_1}$ in 
\begin{equation*}
f_0^{\prime}
=
f_0(f_1-f_2+f_3-f_4)+\alpha_0,
\end{equation*}
we have $0=a_{\infty,n_0}b_{\infty,n_1}, $ 
which is contradiction. 
We can prove that neither of cases (ii), (iii) occurs in the same way. 
\par
We next deal with case (iv). 
We show that $n_0=n_1=n_2=1.$ For this purpose, 
we set $n_0=n_1=n_2=n\geq 1.$
By comparing the coefficients of  term $t^{2n}$ in 
\begin{equation*}
\begin{cases}
f_0^{\prime}=f_0(f_1-f_2+f_3-f_4)+\alpha_0 \\
f_1^{\prime}=f_1(f_2-f_3+f_4-f_0)+\alpha_1 \\
f_2^{\prime}=f_2(f_3-f_4+f_0-f_1)+\alpha_2, 
\end{cases}
\end{equation*}
we have 
\begin{equation*}
\begin{cases}
b_{\infty, n}-c_{\infty, n}=0 \\
c_{\infty, n}-a_{\infty, n}=0 \\
a_{\infty, n}-b_{\infty, n}=0,
\end{cases}
\end{equation*}
which implies that $a_{\infty,n}=b_{\infty,n}=c_{\infty,n}\neq 0.$  
Since $\sum_{k=0}^4 f_k=t,$ 
it follows that 
\begin{equation*}
n_0=n_1=n_2=n=1, \,\,a_{\infty,1}=b_{\infty,1}=c_{\infty,1}=1/3.
\end{equation*}
\end{proof}

\begin{proposition}
\label{prop:t=inf-f0,f1,f2-(2)}
{\it
Suppose that 
for $A^{(1)}_4(\alpha_j)_{0\leq j \leq 4},$ 
there exists a solution such that 
for some $i=0,1,2,3,4,$ 
$f_i,f_{i+1},f_{i+2}$ have a pole at $t=\infty$ and $f_{i+3},f_{i+4}$ are both holomorphic at $t=\infty.$ 
Then, 
\begin{equation*}
\begin{cases}
f_i 
= 
1/3t 
+ 
( 
\alpha_{i+1}-\alpha_{i+2}-3\alpha_{i+3}-\alpha_{i+4} 
) 
t^{-1} + \cdots \\
f_{i+1} 
= 
1/3t 
+ 
( 
\alpha_{i+2}-\alpha_i-\alpha_{i+3}+\alpha_{i+4} 
) 
t^{-1} + \cdots \\
f_{i+2} 
= 
1/3t 
+ 
( 
\alpha_i-\alpha_{i+1}+\alpha_{i+3}+3\alpha_{i+4} 
) t^{-1}
+ \cdots \\
f_{i+3} 
= 
3\alpha_{i+3} 
t^{-1} + \cdots \\
f_{i+4} 
= 
-3\alpha_{i+4} t^{-1} + \cdots.
\end{cases}
\end{equation*}
}
\end{proposition}

\begin{proof}
By $\pi,$ we assume that $f_0,f_1,f_2$ have a pole at $t=\infty.$ 
Then 
it follows 
from Proposition \ref{prop:t=inf-f0,f1,f2-(1)} and its proof that 
\begin{equation}
\begin{cases}
f_0=1/3 t +\sum^{k= \infty}_{k=0}a_{\infty,-k}t^{-k}, \,\,
f_1=1/3 t +\sum^{k= \infty}_{k=0}b_{\infty, -k}t^{-k},  \,\,
f_2=1/3 t +\sum^{k= \infty}_{k=0}c_{\infty,-k}t^{-k},  \\
f_3=\sum^{k=\infty}_{k=0}d_{\infty, -k}t^{-k},  \,\,
f_4=\sum^{k=\infty}_{k=0}e_{\infty, -k}t^{-k}. 
\end{cases}
\end{equation} 
By comparing the coefficients of the term $t$ in 
\begin{equation*}
f_3^{\prime}=f_3(f_4-f_0+f_1-f_2)+\alpha_3,
\end{equation*}  
we then obtain $d_{\infty,0}=0$. 
Moreover, comparing the constant terms in 
\begin{equation*}
f_3^{\prime}=f_3(f_4-f_0+f_1-f_2)+\alpha_3,
\end{equation*}  
we have 
\begin{equation*}
d_{\infty,-1}=3\alpha_3.
\end{equation*}
In the same way, we get 
\begin{equation*}
e_{\infty,0}=0, \, \, e_{\infty,-1}=-3\alpha_4.
\end{equation*}
\par
By comparing the coefficients of the term $t$ in 
\begin{equation*}
\begin{cases}
f_0^{\prime}=f_0(f_1-f_2+f_3-f_4)+\alpha_0 \\
f_1^{\prime}=f_1(f_2-f_3+f_4-f_0)+\alpha_1, 
\end{cases}
\end{equation*}
we obtain
\begin{equation*}
\begin{cases}
b_{\infty,0}-a_{\infty,0}=0 \\
c_{\infty,0}-a_{\infty,0}=0,
\end{cases}
\end{equation*}
respectively. 
Since $\sum_{k=0}^4 f_k=t$ and $d_{\infty,0}=e_{\infty,0}=0,$  
it follows that 
\begin{equation*}
a_{\infty,0}=b_{\infty,0}=c_{\infty,0}=0. 
\end{equation*}
\par
By comparing the constant terms in 
\begin{equation*}
\begin{cases}
f_0^{\prime}=f_0(f_1-f_2+f_3-f_4)+\alpha_0 \\
f_1^{\prime}=f_1(f_2-f_3+f_4-f_0)+\alpha_1 \\
f_2^{\prime}=f_2(f_3-f_4+f_0-f_1)+\alpha_2, 
\end{cases}
\end{equation*}
we get 
\begin{equation*}
\begin{cases}
b_{\infty,-1}-c_{\infty,-1}
=1-3\alpha_0-3\alpha_3-3\alpha_4 \\
c_{\infty,-1}-a_{\infty,-1}
=1-3\alpha_1+3\alpha_3+3\alpha_4 \\
a_{\infty,-1}-b_{\infty,-1}
=
1-3\alpha_2-3\alpha_3-3\alpha_4,
\end{cases}
\end{equation*}
respectively. 
Since $\sum_{k=0}^4 f_k=t$ and $d_{\infty,-1}=3\alpha_3, e_{\infty,-1}=-3\alpha_4,$  
it follows that 
\begin{equation*}
\begin{cases}
a_{\infty,-1}=\alpha_1-\alpha_2-3\alpha_3-\alpha_4, \\
b_{\infty,-1}=-\alpha_0+\alpha_2-\alpha_3+\alpha_4, \\
c_{\infty,-1}=\alpha_0-\alpha_1+\alpha_3+3\alpha_4.
\end{cases}
\end{equation*}

\end{proof}

\begin{proposition}
\label{prop:t=inf-f0,f1,f2-(3)}
{\it
Suppose that 
for $A^{(1)}_4(\alpha_j)_{0\leq j \leq 4},$ 
there exists a solution such that 
for some $i=0,1,2,3,4,$ 
$f_i,f_{i+1},f_{i+2}$ have a pole at $t=\infty$ and $f_{i+3},f_{i+4}$ are both holomorphic at $t=\infty.$ 
It is then unique. 
}
\end{proposition}

\begin{proof}
By $\pi$ and Proposition \ref{prop:t=inf-f0,f1,f2-(2)}, 
we can assume that  
\begin{equation*}
\begin{cases}
f_0=1/3t+a_{\infty,-1}t^{-1}+\sum^{\infty}_{k=2}a_{\infty,-k}t^{-k}, \\
f_1=1/3t+b_{\infty, -1}t^{-1}+\sum^{\infty}_{k=2 }b_{\infty, -k}t^{-k},  \\
f_2=1/3t+c_{\infty, -1}t^{-1}+\sum^{\infty}_{k=2 }c_{\infty,-k}t^{-k},  \\
f_3=d_{\infty, -1}t^{-1}+\sum^{\infty}_{k=2 }d_{\infty, -k}t^{-k},  \\
f_4=e_{\infty, -1}t^{-1}+\sum^{\infty}_{k=2 }e_{\infty, -k}t^{-k}, 
\end{cases}
\end{equation*}
where $a_{\infty, -1}, b_{\infty, -1}, c_{\infty, -1}, d_{\infty, -1}, e_{\infty, -1}$ have been determined in Proposition \ref{prop:t=inf-f0,f1,f2-(2)}. 
By comparing the coefficients of the terms $t^{-k} \,\,(k \geq 1)$ in
\begin{equation*}
\begin{cases}
f_3^{\prime}=f_3(f_4-f_0+f_1-f_2)+\alpha_3 \\
f_4^{\prime}=f_4(f_0-f_1+f_2-f_3)+\alpha_4,  
\end{cases}
\end{equation*}
we obtain
\begin{equation*}
\begin{cases}
d_{\infty, -(k+1)}
=
3(k-1)d_{\infty, -(k-1)} +
3 \sum_{l,m} d_{\infty, -l}(e_{\infty, -m}-a_{\infty, -m}+b_{\infty, -m}-c_{\infty, -m})  \\
e_{\infty, -(k+1)}
=
-3(k-1)e_{\infty, -(k-1)}
-3 \sum_{l,m}  e_{\infty, -l}
(a_{\infty, -m}-b_{\infty, -m}+c_{\infty, -m}-d_{\infty, -m}),
\end{cases}
\end{equation*}
where the sums extend over the positive integers $l$ and $m$ such that $l+m=k.$ 
Therefore, 
the coefficients 
$d_{\infty, -k},e_{\infty, -k} \,\,(k \geq 1)$ 
are inductively determined.
\par
By comparing the coefficients of the terms $t^{-k} \,\,(k\geq 1)$ in
\begin{equation*}
\begin{cases}
f_0^{\prime}=f_0(f_1-f_2+f_3-f_4)+\alpha_0 \\ 
f_1^{\prime}=f_1(f_2-f_3+f_4-f_0)+\alpha_1 \\  
f_2^{\prime}=f_2(f_3-f_4+f_0-f_1)+\alpha_2,    
\end{cases}
\end{equation*}
we have 
\begin{equation*}
\begin{cases}
c_{\infty, -(k+1)}-b_{\infty, -(k+1)}
=
-3(k-1)a_{\infty, -(k-1)}
+3\sum_{l,m} a_{\infty, -l}(b_{\infty, -m}-c_{\infty, -m}+d_{\infty, -m}-e_{\infty, -m})  \\
a_{\infty, -(k+1)}-c_{\infty, -(k+1)}
=
-3(k-1)b_{\infty, -(k-1)}
+3\sum_{l,m} b_{\infty, -l}(c_{\infty, -m}-d_{\infty, -m}+e_{\infty, -m}-a_{\infty, -m})  \\
b_{\infty, -(k+1)}-a_{\infty, -(k+1)}
=
-3(k-1)c_{\infty,-(k-1)}
+3\sum_{l,m} c_{\infty, -l}(d_{\infty, -m}-e_{\infty, -m}+a_{\infty, -m}-b_{\infty, -m}),
\end{cases}
\end{equation*}
where the sums extend over the positive integers $l$ and $m$ such that $l+m=k.$ 
Since $\sum_{i=0}^4 f_i=t,$ 
it follows that 
\begin{equation*}
a_{\infty, -(k+1)}+b_{\infty, -(k+1)}+c_{\infty, -(k+1)}
=-d_{\infty, -(k+1)}-e_{\infty, -(k+1)}.
\end{equation*}
Therefore, 
the coefficients 
$a_{\infty, -k},b_{\infty, -k}, c_{\infty, -k} \,\,(k \geq 1)$ 
are inductively determined.

\end{proof}

\begin{corollary}
\label{coro:t=inf-three(1)-identically-zero}
{\it
Suppose that 
for $A^{(1)}_4(\alpha_j)_{0\leq j \leq 4},$ 
there exists a solution such that 
for some $i=0,1,2,3,4,$ 
$f_i,f_{i+1},f_{i+2}$ have a pole at $t=\infty$ and $f_{i+3},f_{i+4}$ are both holomorphic at $t=\infty.$ 
Then, 
\begin{equation*}
\begin{cases}
f_3 \equiv 0 \,\, \text{if and only if} \,\,\alpha_3=0, \\
f_4 \equiv 0 \,\,  \text{if and only if} \,\,\alpha_4=0. 
\end{cases}
\end{equation*}
}
\end{corollary}

\subsubsection{The case where $f_i,f_{i+1},f_{i+3}$ have a pole at $t=\infty$}

\begin{proposition}
\label{prop:t=inf-f0,f1,f3-(1)}
{\it
Suppose that 
for $A^{(1)}_4(\alpha_j)_{0\leq j \leq 4},$ 
there exists a solution such that 
for some $i=0,1,2,3,4,$ 
$f_i,f_{i+1},f_{i+3}$ have a pole at $t=\infty$ and $f_{i+2},f_{i+4}$ are both holomorphic at $t=\infty.$ 
Then, $n_{i}=n_{i+1}=n_{i+3}=1.$
}
\end{proposition}

\begin{proof}
By $\pi,$ we can assume that $f_0,f_1,f_3$ have a pole at $t=\infty$. 
Since $\sum_{k=0}^4 f_k=t,$ one of the following four cases occurs.
\begin{align*} 
&\mathrm{(i)} \quad n_0=n_1>n_3, &     &\mathrm{(ii)} \quad n_1=n_3>n_0, \\
&\mathrm{(iii)} \quad n_3=n_1>n_0, &   &\mathrm{(iv)} \quad n_0=n_1=n_3.
\end{align*}
\par
We first treat case (i). 
We then note that $a_{\infty,n_0}b_{\infty, n_1}\neq 0.$ 
By comparing the coefficients of the term $t^{n_0+n_1}$ in 
$$
f_0^{\prime}=f_0(f_1-f_2+f_3-f_4)+\alpha_0,
$$
we have $a_{\infty,n_0}b_{\infty, n_1}=0,$ which is a contradiction. 
In the same way, 
we can prove that neither case (ii) nor (iii) occur. 
\par
We next deal with case (iv) and show that $n_0=n_1=n_3=1.$ 
For this purpose, 
we assume that $n_0=n_1=n_3=n\geq 2.$ 
By comparing the coefficients of the term $t^{2n} $ in 
\begin{equation*}
f_0^{\prime}=f_0(f_1-f_2+f_3-f_4)+\alpha_0,
\end{equation*}
we then get 
$b_{\infty, n}+d_{\infty, n}=0.$ 
Since $\sum_{k=0}^4 f_k=t,$ it follows that $a_{\infty,n_0}=0,$ which is a contradiction. 
Therefore, we obtain 
\begin{equation*}
n_0=n_1=n_3=1.
\end{equation*}

\end{proof}

\begin{proposition}
\label{prop:t=inf-f0,f1,f3-(2)}
{\it
Suppose that 
for $A^{(1)}_4(\alpha_j)_{0\leq j \leq 4},$ 
there exists a solution such that 
for some $i=0,1,2,3,4,$ 
$f_i,f_{i+1},f_{i+3}$ all have a pole at $t=\infty$ and $f_{i+2},f_{i+4}$ are both holomorphic at $t=\infty.$ 
Then, 
\begin{equation*}
\begin{cases}
f_i 
= 
t 
+ 
(1-\alpha_i) t^{-1} + \cdots \\
f_{i+1} 
= 
t 
+ 
( 
1 
-\alpha_{i+1}-2\alpha_{i+2}+2\alpha_{i+4}
) 
t^{-1} + \cdots \\
f_{i+2} 
= 
\alpha_{i+2} 
t^{-1} + \cdots \\
f_{i+3} 
= 
- t 
+ 
( 
-1-\alpha_{i+3}-2\alpha_{i+4} 
) 
t^{-1}
+ \cdots \\
f_{i+4} 
= 
-\alpha_{i+4} 
t^{-1} + \cdots.
\end{cases}
\end{equation*}
}
\end{proposition}

\begin{proof}
By $\pi,$ we can assume that $f_0,f_1,f_3$ have a pole at $t=\infty.$ 
Then, 
it follows 
from Proposition \ref{prop:t=inf-f0,f1,f3-(1)} that 
\begin{equation*}
\begin{cases}
f_0=a_{\infty, 1}t+a_{\infty,0}+a_{\infty,-1}t^{-1}+\cdots, \\
f_1=b_{\infty, 1}t+b_{\infty,0}+b_{\infty,-1}t^{-1}+\cdots,   \,\,
f_2=c_{\infty,0}+c_{\infty,-1}t^{-1}+\cdots,  \\
f_3=d_{\infty, 1}t+d_{\infty,0}+d_{\infty,-1}t^{-1}+\cdots,  \,\,
f_4=e_{\infty,0}+e_{\infty,-1}t^{-1}+\cdots.
\end{cases}
\end{equation*}
\par
By comparing the coefficients of the term $t^2$ in 
\begin{equation*}
\begin{cases}
f_0^{\prime}=f_0(f_1-f_2+f_3-f_4)+\alpha_0 \\
f_1^{\prime}=f_1(f_2-f_3+f_4-f_0)+\alpha_1, 
\end{cases}
\end{equation*}
we have 
\begin{equation*}
\begin{cases}
b_{\infty,1}+d_{\infty,1}=0 \\
a_{\infty,1}+d_{\infty,1}=0,
\end{cases}
\end{equation*}
respectively. 
Since $\sum_{k=0}^4 f_k=t,$ 
it follows that 
\begin{equation*}
a_{\infty,1}=b_{\infty, 1}=1,\,\,d_{\infty,1}=-1.
\end{equation*}
By comparing the coefficients of the term $t$ in 
\begin{equation*}
\begin{cases}
f_2^{\prime}=f_2(f_3-f_4+f_0-f_1)+\alpha_2, \\
f_4^{\prime}=f_4(f_0-f_1+f_2-f_3)+\alpha_4, 
\end{cases}
\end{equation*}
we then obtain 
\begin{equation*}
c_{\infty,0}=e_{\infty,0}=0,
\end{equation*}
respectively. 
Moreover, comparing the constant terms in 
\begin{equation*}
\begin{cases}
f_2^{\prime}=f_2(f_3-f_4+f_0-f_1)+\alpha_2, \\
f_4^{\prime}=f_4(f_0-f_1+f_2-f_3)+\alpha_4, 
\end{cases}
\end{equation*} 
we have
\begin{equation*}
c_{\infty,-1}=\alpha_2, \,\,e_{\infty,-1}=-\alpha_4,
\end{equation*} 
respectively. 
\par
By comparing the coefficients of the term $t$ in 
\begin{equation*}
\begin{cases}
f_0^{\prime}=f_0(f_1-f_2+f_3-f_4)+\alpha_0, \\
f_1^{\prime}=f_1(f_2-f_3+f_4-f_0)+\alpha_1, 
\end{cases}
\end{equation*}
we obtain
\begin{equation*}
\begin{cases}
b_{\infty,0}+d_{\infty,0}=0, \\
a_{\infty,0}+d_{\infty,0}=0,
\end{cases}
\end{equation*} 
respectively. 
Since $\sum_{k=0}^4 f_k=t$ and $c_{\infty,0}=e_{\infty,0}=0,$ it follows that 
\begin{equation*}
a_{\infty,0}=b_{\infty,0}=d_{\infty,0}=0.
\end{equation*}
By comparing the constant terms in 
\begin{equation*}
\begin{cases}
f_0^{\prime}=f_0(f_1-f_2+f_3-f_4)+\alpha_0 \\
f_1^{\prime}=f_1(f_2-f_3+f_4-f_0)+\alpha_1, 
\end{cases}
\end{equation*}
we then have
\begin{equation*}
\begin{cases}
a_{\infty,-1}=-2\alpha_2+2\alpha_4+\alpha_0-1 \\
d_{\infty,-1}=-\alpha_0+\alpha_1+3\alpha_2-3\alpha_4.
\end{cases}
\end{equation*}
respectively. 
Since $\sum_{k=0}^4 f_k=t,$ 
it follows that 
\begin{equation*}
b_{\infty,-1}=-\alpha_1-2\alpha_2+2\alpha_4+1.
\end{equation*}

\end{proof}

\begin{proposition}
\label{prop:t=inf-f0,f1,f3-(3)}
{\it
Suppose that 
for $A^{(1)}_4(\alpha_j)_{0\leq j \leq 4},$ 
there exists a solution such that 
for some $i=0,1,2,3,4,$ 
$f_i,f_{i+1},f_{i+3}$ all have a pole at $t=\infty$ and $f_{i+2},f_{i+4}$ are both holomorphic at $t=\infty.$ 
It is then unique. 
}
\end{proposition}

\begin{proof}
By $\pi,$ 
we can assume that $f_0,f_1,f_3$ have a pole at $t=\infty.$ 
From Proposition \ref{prop:t=inf-f0,f1,f3-(2)}, 
we can set 
\begin{equation*}
\begin{cases}
f_0=t+a_{\infty, -1}t^{-1}+\sum^{\infty}_{k=2 }a_{\infty, -k}t^{-k}, \\
f_1=t+b_{\infty, -1}t^{-1}+\sum^{\infty}_{k=2 }b_{\infty, -k}t^{-k},  \\
f_2=c_{\infty,-1}t^{-1}+\sum^{\infty}_{k=2 }c_{\infty, -k}t^{-k},  \\
f_3=-t+d_{\infty, -1}t^{-1}+\sum^{\infty}_{k=2 }d_{\infty, -k}t^{-k},  \\
f_4=e_{\infty, -1}t^{-1}+\sum^{\infty}_{k=2 }e_{\infty, -k}t^{-k}, 
\end{cases}
\end{equation*}
where $a_{\infty, -1}, b_{\infty, -1}, c_{\infty, -1}, d_{\infty, -1}, e_{\infty, -1}$ have been determined in Proposition \ref{prop:t=inf-f0,f1,f3-(2)}. 
\par
By comparing the coefficients of the terms $t^{-k} \,\,(k \geq 1)$ in
\begin{equation*}
\begin{cases}  
f_2^{\prime}=f_2(f_3-f_4+f_0-f_1)+\alpha_2 \\  
f_4^{\prime}=f_4(f_0-f_1+f_2-f_3)+\alpha_4,  
\end{cases}
\end{equation*}
we get
\begin{equation*}
\begin{cases}
c_{\infty, -(k+1)}=
c_{\infty, -(k-1)}(k-1)
+\sum_{l,m} c_{\infty, -l}(d_{\infty, -m}-e_{\infty, -m}+a_{\infty, -m}-b_{\infty, -m}) \\
e_{\infty, -(k+1)}
=
-e_{\infty, -(k-1)}(k-1)
-\sum_{l,m} e_{\infty, -l}(a_{\infty, -m}-e_{\infty, -m}+c_{\infty, -m}-d_{\infty, -m}), 
\end{cases}
\end{equation*}
where the sums extend over the positive integers $l$ and $m$ such that $k+m=k.$ 
Therefore, 
the coefficients 
$c_{\infty, -k},e_{\infty, -k} \,\,(k \geq 1)$ 
are inductively determined. 
\par
By comparing the coefficients of the terms $t^{-k} \,\,(k \geq 1)$ in 
\begin{equation*}
\begin{cases}
f_0^{\prime}=f_0(f_1-f_2+f_3-f_4)+\alpha_0 \\
f_1^{\prime}=f_1(f_2-f_3+f_4-f_0)+\alpha_1 \\    
f_3^{\prime}=f_3(f_4-f_0+f_1-f_2)+\alpha_3,
\end{cases}
\end{equation*}
we get
\begin{equation*}
\begin{cases}
b_{\infty, -(k+1)}+d_{\infty, -(k+1)}
=c_{\infty,-(k+1)}+e_{\infty, -(k+1)}
-a_{\infty, -(k-1)}(k-1) \\
\hspace{40mm} -\sum_{l,m} a_{\infty, -l}(b_{\infty, -m}-c_{\infty, -m}+d_{\infty, -m}-e_{\infty, -m}) \\
-d_{\infty, -(k+1)}-a_{\infty, -(k+1)}
=-c_{\infty, -(k+1)}-e_{\infty, -(k+1)}
-b_{\infty, -(k-1)}(k-1) \\
\hspace{40mm}  -\sum_{l,m} b_{\infty, -l}(c_{\infty, -m}-d_{\infty, -m}+e_{\infty, -m}-a_{\infty, -m}) \\
-a_{\infty, -(k+1)}+b_{\infty, -(k+1)}
=
-e_{\infty, -(k+1)}+c_{\infty, -(k+1)}
-d_{k+1}(k-1) \\
\hspace{40mm}  +\sum_{l,m} d_{\infty, -l}(e_{\infty, -m}-a_{\infty, -m}+b_{\infty, -m}-c_{\infty, -m}),
\end{cases}
\end{equation*} 
where the sums extend over the positive integers $l$ and $m$ such that $l+m=k.$ 
Since $\sum_{j=0}^4 f_j=t,$ 
it follows that 
\begin{equation*}
a_{\infty, -(k+1)}+b_{\infty, -(k+1)}+d_{\infty, -(k+1)}=-c_{\infty,-(k+1)}-e_{\infty, -(k+1)}. 
\end{equation*}
Therefore, 
the coefficients 
$a_{\infty, -k},b_{\infty, -k}, d_{\infty, -k} \,\,(k \geq 1)$ 
are inductively determined.

\end{proof}

From the proof of Proposition \ref{prop:t=inf-f0,f1,f3-(3)}, 
we can show the following proposition:

\begin{corollary}
\label{coro:t=inf-f0,f1,f3-identically-zero}
{\it
Suppose that 
for $A^{(1)}_4(\alpha_j)_{0\leq j \leq 4},$ 
there exists a solution such that 
for some $i=0,1,2,3,4,$ 
$f_i,f_{i+1},f_{i+3}$ all have a pole at $t=\infty$ and $f_{i+2},f_{i+4}$ are both holomorphic at $t=\infty.$ 
Then, 
\begin{equation*}
\begin{cases}
f_{i+2} \equiv 0 \,\,\text{if and only if} \,\,\alpha_{i+2}=0, \\
f_{i+4} \equiv 0 \,\, \text{if and only if} \,\,\alpha_{i+4}=0. 
\end{cases}
\end{equation*} 
}
\end{corollary}

\subsection{The case where four of $f_i \,(\leq i \leq 4)$ have a pole at $t=\infty$}
By $\pi,$ 
we have only to consider the following case: 
for some $i=0,1,2,3,4,$ 
$f_i,$ $f_{i+1},$ $f_{i+2},$ $f_{i+3}$ all have a pole at $t=\infty$ and 
$f_{i+4}$ is holomorphic at $t=\infty.$

\begin{proposition}
\label{prop:t=inf-four}
{\it
For $A_4^{(1)}(\alpha_j)_{0\leq j \leq 4},$ 
there exists no solution such that 
$f_i, f_{i+1}, f_{i+2}, f_{i+3}$ all have a pole at $t=\infty$ and 
$f_{i+4}$ is holomorphic at $t=\infty.$ 
}
\end{proposition}

\begin{proof}
By $\pi,$ we can assume that 
$f_0,f_1,f_2,f_3$ have a pole at $t=\infty.$ 
Since $\sum_{k=0}^4 f_k=t,$ 
it follows that one of the following eleven cases occurs: 
\begin{align*}
&\mathrm{(i)} \quad  
n_0=n_1>
\left\{
\begin{array}{l}
n_2 \\
n_3 
\end{array}
\right\}
\geq 1 & 
&\mathrm{(ii)}  \quad
n_0=n_2>
\left\{
\begin{array}{l}
n_1 \\
n_3
\end{array}
\right\}
\geq 1 \\
&\mathrm{(iii)} \quad
n_0=n_3>
\left\{
\begin{array}{l}
n_1 \\
n_2 
\end{array}
\right\}
\geq 1 & 
&\mathrm{(iv)}  \quad
n_1=n_2>
\left\{
\begin{array}{l}
n_0 \\
n_3 
\end{array}
\right\}
\geq 1
\\
&\mathrm{(v)}  \quad
n_1=n_3>
\left\{
\begin{array}{l}
n_0 \\
n_2 
\end{array}
\right\}
\geq 1 & 
&\mathrm{(vi)}  \quad
n_2=n_3>
\left\{
\begin{array}{l}
n_0 \\
n_1 
\end{array}
\right\}
\geq 1 \\
&\mathrm{(vii)}   \quad n_0=n_1=n_2>n_3 \geq 1  & 
&\mathrm{(viii)}   \quad n_0=n_1=n_3>n_2 \geq 1  \\
&\mathrm{(ix)}  \quad n_1=n_2=n_3>n_0 \geq 1  & 
&\mathrm{(x)}   \quad  n_2=n_3=n_0>n_1 \geq 1  \\
&\mathrm{(xi)}  \quad n_0=n_1=n_2=n_3 \geq 1. & &
\end{align*} 
\par
We first treat case (i). Then, 
we note that $a_{\infty, n_0}b_{\infty, n_1}\neq 0.$ 
Therefore, 
comparing the coefficients of the term $t^{n_0+n_1}$ in 
$$
f_0^{\prime}=f_0(f_1-f_2+f_3-f_4)+\alpha_0
$$
we have $0=a_{\infty,n_0}b_{\infty, n_1},$ which is a contradiction. 
In the same way, 
we can prove that 
none of cases (ii), (iii), (iv), (v), (vi) happens. 
\par
We next deal with cases (vii). 
For this purpose, 
comparing the coefficients of the terms $t^{n_0+n_1}=t^{n_0+n_2},$ $t^{n_1+n_2}=t^{n_1+n_0}$ in 
\begin{equation*}
\begin{cases}
f_0^{\prime}=f_0(f_1-f_2+f_3-f_4)+\alpha_0 \\
f_1^{\prime}=f_1(f_2-f_3+f_4-f_0)+\alpha_1, \\
\end{cases}
\end{equation*}
we get 
\begin{equation*}
\begin{cases}
b_{\infty, n_1}-c_{\infty, n_2}=0 \\
c_{\infty, n_2}-a_{\infty, n_0}=0.
\end{cases}
\end{equation*}
Since $\sum_{k=0}^4 f_k=t,$ 
it follows that 
\begin{equation*}
a_{\infty, n_0}=b_{\infty, n_1}=c_{\infty, n_2}=0,
\end{equation*}
which is a contradiction. 
In the same way, 
we can show that case (ix) does not occur. 
\par
We deal with case (viii): $n_0=n_1=n_3=n>n_2 \geq 1$ 
For this purpose, 
comparing the coefficients of the term $t^{n+n_2}$ in 
\begin{equation*}
f_2^{\prime}=f_2(f_3-f_4+f_0-f_1)+\alpha_2, 
\end{equation*}
we have
\begin{equation*}
d_{\infty, n}+a_{\infty, n}-b_{\infty, n}=0.
\end{equation*}
Since $\sum_{k=0}^4 f_k=t,$ 
it follows that 
\begin{equation*}
b_{\infty, n}=0,
\end{equation*}
which is a contradiction. 
In the same way, 
we can show that case (x) does not occur. 
\par
We last treat case (xi): $n_0=n_1=n_2=n_3=n \geq 1.$ 
For this purpose, 
comparing the coefficients of the term $t^{2n}$ in
\begin{equation*}
\begin{cases}
f_0^{\prime}=f_0(f_1-f_2+f_3-f_4)+\alpha_0 \\ 
f_1^{\prime}=f_1(f_2-f_3+f_4-f_0)+\alpha_1 \\  
f_2^{\prime}=f_2(f_3-f_4+f_0-f_1)+\alpha_2 \\  
f_3^{\prime}=f_3(f_4-f_0+f_1-f_2)+\alpha_3, 
\end{cases}
\end{equation*}
we obtain 
\begin{equation*}
(*)
\begin{cases}
b_{\infty, n}-c_{\infty, n}+d_{\infty, n} &=0 \\
c_{\infty, n}-d_{\infty, n}-a_{\infty, n} &=0 \\
d_{\infty, n}+a_{\infty, n}-b_{\infty, n} &=0 \\
-a_{\infty, n}+b_{\infty, n}-c_{\infty, n} &=0,
\end{cases}
\end{equation*}
respectively.
We assume that $n_0=n_1=n_2=n_3=n \geq 2.$
Since $\sum_{k=0}^4 f_k=t,$ 
it follows that $a_{\infty,n}+b_{\infty, n}+c_{\infty,n}+d_{\infty, n}=0,$ 
which implies that 
\begin{equation*}
a_{\infty, n}=-2c_{\infty, n}, \,\, b_{\infty, n}=c_{\infty, n}, \,\,d_{\infty, n}=3c_{\infty, n}.
\end{equation*}
Since $\sum_{k=0}^4 f_k=t,$
it follows that 
\begin{equation*}
c_{\infty, n}=0,
\end{equation*}
which is a contradiction.
\par
We assume that $n_0=n_1=n_2=n_3 = 1.$ 
The first equation of $(*)$ implies that 
\begin{equation*}
a_{\infty,1}+2c_{\infty, 1}=1,
\end{equation*}
because $\sum_{k=0}^4 f_k=t.$ 
The second and third equations of $(*)$ then imply that 
\begin{equation*}
d_{\infty,1}=3c_{\infty, 1}-1, \,\,b_{\infty,1}=c_{\infty,1}.
\end{equation*}  
Since $\sum_{k=0}^4 f_k=t,$ it follows that 
\begin{equation*}
1 = a_{\infty,1} + b_{\infty,1} +c_{\infty,1} +d_{\infty,1} =3c_{\infty,1}.
\end{equation*}
Therefore, 
we obtain 
\begin{equation*}
c_{\infty,1}= 1/3, \,\, d_{\infty,1}=0, 
\end{equation*}
which is a contradiction.

\end{proof}

\subsection{The case where all of $f_i \,(\leq i \leq 4)$ have a pole at $t=\infty$}

\begin{proposition}
\label{prop:t=inf-all-(1)}
{\it
Suppose that 
for $A_4^{(1)}(\alpha_j)_{0\leq j \leq 4},$ 
there exists a solution such that 
$f_0, f_1,f_2,f_3,f_4$ all have a pole at $t=\infty.$ 
Then, 
$n_0=n_1=n_2=n_3=n_4=1.$ 
}
\end{proposition}

\begin{proof}
Since $\sum_{k=0}^4 f_k=t,$ 
we can assume that 
one of the following twelve cases occurs.
\begin{align*}
&\mathrm{(i)} \quad n_0=n_1> 
\left\{
\begin{array}{l}
n_2 \\
n_3 \\
n_4
\end{array}
\right\}
\geq 1, & 
&\mathrm{(ii)} \quad n_0=n_2> 
\left\{
\begin{array}{l}
n_1 \\
n_3 \\
n_4
\end{array}
\right\}
\geq 1,  \\
&\mathrm{(iii)} \quad n_0=n_3> 
\left\{
\begin{array}{l}
n_1 \\
n_2 \\
n_4
\end{array}
\right\}
\geq 1, & 
&\mathrm{(iv)} \quad n_0=n_4> 
\left\{
\begin{array}{l}
n_1 \\
n_2 \\
n_3
\end{array}
\right\}
\geq 1,  \\
&\mathrm{(v)} \quad 
n_0=n_1=n_2>
\left\{
\begin{array}{l}
n_3 \\
n_4
\end{array}
\right\}
\geq 1, &  
&\mathrm{(vi)}
\quad 
n_0=n_1=n_3>
\left\{
\begin{array}{l}
n_2 \\
n_4
\end{array}
\right\}
\geq 1,  \\
&\mathrm{(vii)} \quad 
n_0=n_1=n_4>
\left\{
\begin{array}{l}
n_2 \\
n_3
\end{array}
\right\}
\geq 1, & 
&\mathrm{(viii)}
\quad 
n_0=n_2=n_3>
\left\{
\begin{array}{l}
n_1 \\
n_4
\end{array}
\right\}
\geq 1,  \\
\end{align*}
\begin{align*}
&\mathrm{(ix)}
\quad
n_0=n_2=n_4>
\left\{
\begin{array}{l}
n_1 \\
n_3
\end{array}
\right\}
\geq 1, & 
&\mathrm{(x)}
\quad
n_0=n_3=n_4>
\left\{
\begin{array}{l}
n_1 \\
n_2
\end{array}
\right\}
\geq 1, 
\\
&\mathrm{(xi)}
\quad
n_0=n_1=n_2=n_3>
n_4
\geq 1, & 
&\mathrm{(xii)}
\quad
n_0=n_1=n_2=n_3=
n_4
\geq 1. 
\end{align*}
\par
We first treat case (i). 
Then, 
comparing the coefficients of the term $t^{n_0+n_1}$ in 
$$
f_0^{\prime}=f_0(f_1-f_2+f_3-f_4)+\alpha_0,
$$
we have $0=a_{\infty,n_0}b_{\infty,n_1},$ 
which is a contradiction. 
In the same way, 
we can prove that 
none of cases (ii), (iii), (iv) occurs. 
\par
We next deal with case (v). 
For this purpose, 
comparing the coefficients of the terms $t^{n_0+n_1}=t^{n_0+n_2},$ $t^{n_1+n_2}=t^{n_1+n_0}$ in 
\begin{equation*}
\begin{cases}
f_0^{\prime}=f_0(f_1-f_2+f_3-f_4)+\alpha_0, \\ 
f_1^{\prime}=f_1(f_2-f_3+f_4-f_0)+\alpha_1, 
\end{cases}
\end{equation*}  
we have $b_{\infty,n_1}-c_{\infty,n_2}=0,$ $c_{\infty,n_2}-a_{\infty,n_0}=0,$ 
respectively. 
Since $\sum_{k=0}^4 f_k,$ 
it follows that $a_{\infty,n_0}+b_{\infty,n_1}+c_{\infty,n_2}=0,$ 
which implies that $a_{\infty,n_0}=b_{\infty,n_1}=c_{\infty,n_2}=0.$ 
This is impossible. 
In the same way, 
we can show that 
none of cases (v), $\ldots,$ (x) occurs.  
\par
We treat case (xi): $n_0=n_1=n_2=n_3=n>n_4\geq 1.$ 
By comparing the coefficients of the term $t^{2n},$ or $t^{n+n_4}$ in
\begin{equation*}
\begin{cases}
f_0^{\prime}=f_0(f_1-f_2+f_3-f_4)+\alpha_0 \\ 
f_1^{\prime}=f_1(f_2-f_3+f_4-f_0)+\alpha_1 \\  
f_2^{\prime}=f_2(f_3-f_4+f_0-f_1)+\alpha_2 \\  
f_3^{\prime}=f_3(f_4-f_0+f_1-f_2)+\alpha_3 \\
f_4^{\prime}=f_4(f_0-f_1+f_2-f_3)+\alpha_4,  
\end{cases}
\end{equation*}
we then have
\begin{equation*}
(*)
\begin{cases}
b_{\infty, n} -c_{\infty, n}+d_{\infty, n} &=0 \\
c_{\infty, n}-d_{\infty, n}-a_{\infty, n}  &=0  \\
d_{\infty, n}+a_{\infty, n}-b_{\infty, n}  &=0 \\
-a_{\infty, n}+b_{\infty, n}-c_{\infty, n} &=0 \\
a_{\infty, n}-b_{\infty, n}+c_{\infty, n}-d_{\infty, n}  &=0.
\end{cases}
\end{equation*}
Since $\sum_{k=0}^4 f_k=t,$ 
it follows that 
\begin{equation}
\label{relation-(1)}
a_{\infty, n}+b_{\infty, n}+c_{\infty, n}+d_{\infty, n}=0.
\end{equation}
Then, 
the first equation of $(*)$ implies that 
$
a_{\infty, n}=-2c_{\infty, n}.
$
The second and third equations of $(*)$ imply that 
$
d_{\infty, n}=3c_{\infty, n}, \,\,b_{\infty, n}=c_{\infty, n},
$ 
respectively. 
Therefore, 
the equation (\ref{relation-(1)}) implies that 
$c_{\infty, n}=0,$ which is a contradiction.
\par
We last treat case (xii): $n_0=n_1=n_2=n_3=n_4=m\geq 1.$ 
For this purpose, comparing the coefficients of the term $t^{2m}$ in
\begin{equation*}
\begin{cases}
f_0^{\prime}=f_0(f_1-f_2+f_3-f_4)+\alpha_0 \\ 
f_1^{\prime}=f_1(f_2-f_3+f_4-f_0)+\alpha_1 \\  
f_2^{\prime}=f_2(f_3-f_4+f_0-f_1)+\alpha_2 \\  
f_3^{\prime}=f_3(f_4-f_0+f_1-f_2)+\alpha_3 \\
f_4^{\prime}=f_4(f_0-f_1+f_2-f_3)+\alpha_4,  
\end{cases}
\end{equation*}
we obtain
\begin{equation*}
(**)
\begin{cases}
b_{\infty, m}-c_{\infty, m}+d_{\infty, m}-e_{\infty, m}=0 \\
c_{\infty, m}-d_{\infty, m}+e_{\infty, m}-a_{\infty, m}=0 \\
d_{\infty, m}-e_{\infty,m}+a_{\infty, m}-b_{\infty, m}=0 \\ 
e_{\infty, m}-a_{\infty, m}+b_{\infty, m}-c_{\infty, m}=0 \\
a_{\infty, m}-b_{\infty, m}+c_{\infty, m}-d_{\infty, m}=0,
\end{cases}
\end{equation*}
respectively. 
This system of equations is expressed by 
$A \mathbf{u}=\mathbf{0}, $ 
where 

\begin{equation*}
A=
\left(
\begin{array}{ccccc}
0 & 1 & -1 & 1 & -1 \\
-1 & 0 & 1 & -1 & 1 \\
1 & -1 & 0 & 1 & -1 \\
-1 & 1 & -1 & 0 & 1 \\
1 & -1 & 1 & -1 & 0 \\
\end{array}
\right), \,\,
\mathbf{u}={}^t(a_{\infty,m}, b_{\infty,m},c_{\infty,m},d_{\infty,m},e_{\infty,m}) 
\end{equation*}
Since the rank of $A$ 
is four, 
it follows that 
\begin{equation*}
(a_{\infty,m}, b_{\infty,m},c_{\infty,m},d_{\infty,m},e_{\infty,m}) 
= 
\alpha \,\,
(1,1,1,1,1) \in \ker A, 
\end{equation*}
for some $\alpha \in \C\setminus\{0\}.$ 
Since $\sum_{k=0}^4 f_k=t,$ 
it follows that 
\begin{equation*}
n_0=n_1=n_2=n_3=n_4=1, \,\,
a_{\infty,1}=b_{\infty,1}=c_{\infty,1}=d_{\infty,1}=e_{\infty,1}=1/5.
\end{equation*}

\end{proof}

\begin{proposition}
\label{prop:t=inf-all-(2)}
{\it
Suppose that 
for $A_4^{(1)}(\alpha_j)_{0\leq j \leq 4},$ 
there exists a solution such that 
$f_0, f_1,f_2,f_3,f_4$ all have a pole at $t=\infty.$ 
Then, 
\begin{equation*}
\begin{cases}
f_0 
= 
1/5t 
+ 
( 
3\alpha_1+\alpha_2-\alpha_3-3\alpha_4
) 
t^{-1}  
+ \cdots \\
f_1 
= 
1/5t 
+ 
( 
3\alpha_2+\alpha_3-\alpha_4-3\alpha_0 
) 
t^{-1} + \cdots \\
f_2 
= 
1/5
t 
+ 
( 
3\alpha_3+\alpha_4-\alpha_0-3\alpha_1 
) 
t^{-1} + \cdots \\
f_3 
= 
1/5t 
+ 
( 
3\alpha_4+\alpha_0-\alpha_1-3\alpha_2 
) 
t^{-1} + \cdots \\
f_4 
= 
1/5t 
+ 
( 
3\alpha_0+\alpha_1-\alpha_2-3\alpha_3 
) 
t^{-1} + \cdots. 
\end{cases}
\end{equation*}
}
\end{proposition}

\begin{proof}
From Proposition \ref{prop:t=inf-all-(2)} and its proof, 
it follows that 
\begin{equation}
\begin{cases}
f_0=1/5 t +a_{\infty,0}+a_{\infty,-1}t^{-1}+\cdots, \\
f_1=1/5 t +b_{\infty,0}+b_{\infty,-1}t^{-1}+\cdots,   \\
f_2=1/5 t +c_{\infty,0}+c_{\infty,-1}t^{-1}+\cdots,  \\
f_3=1/5 t +d_{\infty,0}+d_{\infty,-1}t^{-1}+\cdots,  \\
f_4=1/5 t +e_{\infty,0}+e_{\infty,-1}t^{-1}+\cdots.
\end{cases}
\end{equation}
then, 
comparing the coefficients of the term $t$ in
\begin{equation*}
\begin{cases}
f_0^{\prime}=f_0(f_1-f_2+f_3-f_4)+\alpha_0 \\ 
f_1^{\prime}=f_1(f_2-f_3+f_4-f_0)+\alpha_1 \\  
f_2^{\prime}=f_2(f_3-f_4+f_0-f_1)+\alpha_2 \\  
f_3^{\prime}=f_3(f_4-f_0+f_1-f_2)+\alpha_3 \\
f_4^{\prime}=f_4(f_0-f_1+f_2-f_3)+\alpha_4,  
\end{cases}
\end{equation*}
we obtain
\begin{equation*}
\begin{cases}
b_{\infty,0}-c_{\infty,0}+d_{\infty,0}-e_{\infty,0}=0 \\
c_{\infty,0}-d_{\infty,0}+e_{\infty,0}-a_{\infty,0}=0 \\
d_{\infty,0}-e_{\infty,0}+a_{\infty,0}-b_{\infty,0}=0 \\ 
e_{\infty,0}-a_{\infty,0}+b_{\infty,0}-c_{\infty,0}=0 \\
a_{\infty,0}-b_{\infty,0}+c_{\infty,0}-d_{\infty,0}=0. 
\end{cases}
\end{equation*}
Since the rank of 
\begin{equation*}
A=
\left(
\begin{array}{ccccc}
0 & 1 & -1 & 1 & -1 \\
-1 & 0 & 1 & -1 & 1 \\
1 & -1 & 0 & 1 & -1 \\
-1 & 1 & -1 & 0 & 1 \\
1 & -1 & 1 & -1 & 0 \\
\end{array}
\right)
\end{equation*}
is four, 
it follows that 
\begin{equation*}
(a_{\infty,0}, b_{\infty,0},c_{\infty,0},d_{\infty,0},e_{\infty,0}) 
= 
\beta \,\,
(1,1,1,1,1) \in \ker A, 
\end{equation*}
for some $\beta \in \mathbb{C}.$
Since $\sum_{k=0}^4 f_k=t,$ 
it follows that 
\begin{equation*}
a_{\infty,0}=b_{\infty,0}=c_{\infty,0}=d_{\infty,0}=e_{\infty,0}=\beta=0.
\end{equation*}
By comparing the constant terms in
\begin{equation*}
\begin{cases}
f_0^{\prime}=f_0(f_1-f_2+f_3-f_4)+\alpha_0 \\ 
f_1^{\prime}=f_1(f_2-f_3+f_4-f_0)+\alpha_1 \\  
f_2^{\prime}=f_2(f_3-f_4+f_0-f_1)+\alpha_2 \\  
f_3^{\prime}=f_3(f_4-f_0+f_1-f_2)+\alpha_3 \\
f_4^{\prime}=f_4(f_0-f_1+f_2-f_3)+\alpha_4,  
\end{cases}
\end{equation*}
we obtain
\begin{equation*}
\begin{cases}
1=b_{\infty,-1}-c_{\infty,-1}+d_{\infty,-1}+e_{\infty,-1}+5\alpha_0 \\
1=c_{\infty,-1}-d_{\infty,-1}+e_{\infty,-1}-a_{\infty,-1}+5\alpha_1 \\
1=d_{\infty,-1}-e_{\infty,-1}+a_{\infty,-1}-b_{\infty,-1}+5\alpha_2 \\ 
1=e_{\infty,-1}-a_{\infty,-1}+b_{\infty,-1}-c_{\infty,-1}+5\alpha_3 \\
1=a_{\infty,-1}-b_{\infty,-1}+c_{\infty,-1}-d_{\infty,-1}+5\alpha_4. 
\end{cases}
\end{equation*}
Since $\sum_{k=0}^4 f_k=t,$ 
it follows that 
\begin{equation*}
a_{\infty,-1}+b_{\infty,-1}+
c_{\infty,-1}+
d_{\infty,-1}+
e_{\infty,-1}=0.
\end{equation*}
Therefore, 
we get 
\begin{equation*}
\begin{cases}
a_{\infty,-1}=3\alpha_1+\alpha_2-\alpha_3-3\alpha_4, \\
b_{\infty,-1}=3\alpha_2+\alpha_3-\alpha_4-3\alpha_0, \\
c_{\infty,-1}=3\alpha_3+\alpha_4-\alpha_0-3\alpha_1, \\
d_{\infty,-1}=3\alpha_4+\alpha_0-\alpha_1-3\alpha_2, \\
e_{\infty,-1}=3\alpha_0+\alpha_1-\alpha_2-3\alpha_3.  
\end{cases}
\end{equation*}

\end{proof}

\begin{proposition}
\label{prop:t=inf-all-(3)}
{\it
Suppose that 
for $A^{(1)}_4(\alpha_j)_{0\leq j \leq 4},$ 
there exists a solution such that 
$f_0, f_1,f_2,f_3,f_4$ all have a pole at $t=\infty.$ 
It is then unique. 
}
\end{proposition}

\begin{proof}
By Proposition \ref{prop:t=inf-all-(2)}, 
we set 
\begin{equation*}
\begin{cases}
f_0=1/5t+a_{\infty,-1}t^{-1}+\sum_{k=2}^{\infty}a_{\infty, -k}t^{-k}, \\
f_1=1/5t+b_{\infty,-1}t^{-1}+\sum_{k=2}^{\infty}b_{\infty, -k}t^{-k},  \\
f_2=1/5t+c_{\infty,-1}t^{-1}+\sum_{k=2}^{\infty}c_{\infty, -k}t^{-k}, \\  
f_3=1/5t+d_{\infty,-1}t^{-1}+\sum_{k=2}^{\infty}d_{\infty, -k}t^{-k},  \\
f_4=1/5t+e_{\infty,-1}t^{-1}+\sum_{k=2}^{\infty}e_{\infty, -k}t^{-k}, 
\end{cases}
\end{equation*}
where $a_{\infty,-1}, b_{\infty,-1}, c_{\infty,-1}, d_{\infty,-1}, e_{\infty,-1}$ have been determined in Proposition \ref{prop:t=inf-all-(2)}.
\par
By comparing the coefficients of the terms $t^{-k}\,\,(k\geq 1)$ in
\begin{equation*}
\begin{cases}
f_0^{\prime}=f_0(f_1-f_2+f_3-f_4)+\alpha_0 \\ 
f_1^{\prime}=f_1(f_2-f_3+f_4-f_0)+\alpha_1 \\  
f_2^{\prime}=f_2(f_3-f_4+f_0-f_1)+\alpha_2 \\  
f_3^{\prime}=f_3(f_4-f_0+f_1-f_2)+\alpha_3 \\
f_4^{\prime}=f_4(f_0-f_1+f_2-f_3)+\alpha_4,  
\end{cases}
\end{equation*}
we get
\begin{align*}
b_{\infty,-(k+1)}-c_{\infty,-(k+1)}+d_{\infty,-(k+1)}-e_{\infty,-(k+1)}
&=-5(k-1)a_{\infty, -(k-1)}
\\
\hspace{40mm}
&-5\sum_{l,m}  
a_{\infty, -l}(b_{\infty,-m}-c_{\infty,-m}+d_{\infty,-m}-e_{\infty,-m})  
\end{align*}
\begin{align*}
c_{\infty,-(k+1)}-d_{\infty,-(k+1)}+e_{\infty,-(k+1)}-a_{\infty,-(k+1)}
&=-5(k-1)b_{\infty, -(k-1)} \\
\hspace{40mm}&-5\sum_{l,m}
b_{\infty, -l}(c_{\infty,-m}-d_{\infty,-m}+e_{\infty,-m}-a_{\infty,-m})  
\end{align*}
\begin{align*}
d_{\infty,-(k+1)}-e_{\infty,-(k+1)}+a_{\infty,-(k+1)}-b_{\infty,-(k+1)}  
&=-5(k-1)c_{\infty, -(k-1)}            \\
\hspace{40mm}
&
-5\sum_{l,m}
c_{\infty,-l}(d_{\infty,-m}-e_{\infty,-m}+a_{\infty,-m}-b_{\infty,-m})  
\end{align*}
\begin{align*}
e_{\infty,-(k+1)}-a_{\infty,-(k+1)}+b_{\infty,-(k+1)}-c_{\infty,-(k+1)}
&=-5(k-1)d_{\infty,-(k-1)}  \\
\hspace{40mm}
&-5\sum_{l,m}
d_{\infty,-l}(e_{\infty,-m}-a_{\infty,-m}+b_{\infty,-m}-c_{\infty,-m})  
\end{align*}
\begin{align*}
a_{\infty,-(k+1)}-b_{\infty,-(k+1)}+c_{\infty,-(k+1)}-d_{\infty,-(k+1)}
&=-5(k-1)e_{\infty, -(k-1)}  \\
\hspace{40mm}
&-5\sum_{l,m}
e_{\infty,-l}(a_{\infty,-m}-b_{\infty,-m}+c_{\infty,-m}-d_{\infty,-m}) , 
\end{align*}
where the sums extend over the positive integers $l$ and $m$ such that $l+m=k.$ 
Since the rank of 
\begin{equation*}
\left(
\begin{array}{ccccc}
0 & 1 & -1 & 1 & -1 \\
-1 & 0 & 1 & -1 & 1 \\
1 & -1 & 0 & 1 & -1 \\
-1 & 1 & -1 & 0 & 1 \\
1 & -1 & 1 & -1 & 0 \\
\end{array}
\right)
\end{equation*}
is four, 
$b_{\infty,-(k+1)},c_{\infty,-(k+1)},d_{\infty,-(k+1)},e_{\infty,-(k+1)}$ can be expressed 
by 
$$
a_{\infty,-i} \,(1\leq i \leq k+1), \, 
b_{\infty, j}, \, 
c_{\infty, j}, \,
d_{\infty, j} \,
e_j{\infty, j} \,(\leq j \leq k). 
$$
Since $\sum_{k=0}^4 f_k=t,$ 
it follows that 
\begin{equation*}
a_{\infty, -(k+1)}+b_{\infty, -(k+1)}+c_{\infty, -(k+1)}+d_{\infty, -(k+1)}+e_{\infty, -(k+1)}=0.
\end{equation*}
Therefore, 
the coefficients 
$a_{\infty, -k},b_{\infty, -k}, c_{\infty, -k},d_{\infty, -k},e_{\infty, -k} \,\,(k \geq 1)$ 
are inductively determined. 

\end{proof}

\subsection{Summary}

\begin{proposition}
\label{prop:a4inf}
{\it
Suppose that 
for $A_4^{(1)}(\alpha_j)_{0 \leq j \leq 4},$ 
there exists a meromorphic solution at $t=\infty.$ 
One of the following then occurs. 
\newline
Type A (1): \quad for some $i=0,1,2,3,4,$ 
$f_i$ has a pole at $t=\infty$ and $f_{i+1},f_{i+2},f_{i+3},f_{i+4}$ are all holomorphic at $t=\infty,$ and 
$f_i, f_{i+1},f_{i+2},f_{i+3},f_{i+4}$ are uniquely expanded as follows:
\begin{equation*}
\begin{cases}
f_i 
= 
t 
+ 
( 
-\alpha_{i+1}+\alpha_{i+2}-\alpha_{i+3} 
+\alpha_{i+4} 
) 
t^{-1} 
+ 
\cdots \\
f_{i+1} 
= 
\alpha_{i+1} t^{-1} + \cdots \\
f_{i+2} 
= 
-\alpha_{i+2} 
t^{-1} + \cdots \\
f_{i+3} 
= 
\alpha_{i+3} 
t^{-1} + \cdots \\
f_{i+4} 
= 
- 
\alpha_{i+4}  
t^{-1} + \cdots. 
\end{cases}
\end{equation*}
Type A (2): \quad for some $i=0,1,2,3,4,$ $f_i, f_{i+1}, f_{i+3}$ have a pole at $t=\infty$ 
and $f_{i+2},f_{i+4}$ are both holomorphic at $t=\infty,$ 
and 
$f_i, f_{i+1},f_{i+2},f_{i+3},f_{i+4}$ are uniquely expanded as follows:
\begin{equation*}
\begin{cases}
f_i 
= 
t 
+ 
(1-\alpha_i) t^{-1} + \cdots \\
f_{i+1} 
= 
t 
+ 
( 
1 
-\alpha_{i+1}-2\alpha_{i+2}+2\alpha_{i+4}
) 
t^{-1} + \cdots \\
f_{i+2} 
= 
\alpha_{i+2} 
t^{-1} + \cdots \\
f_{i+3} 
= 
- t 
+ 
( 
-1-\alpha_{i+3}-2\alpha_{i+4} 
) 
t^{-1}
+ \cdots \\
f_{i+4} 
= 
-\alpha_{i+4} 
t^{-1} + \cdots.
\end{cases}
\end{equation*}
Type B: \quad for some $i=0,1,2,3,4,$ $f_i, f_{i+1}, f_{i+2}$ have a pole at $t=\infty$ 
and $f_{i+3},f_{i+4}$ are both holomorphic at $t=\infty,$ 
and 
$f_i, f_{i+1},f_{i+2},f_{i+3},f_{i+4}$ are uniquely expanded as follows:
\begin{equation*}
\begin{cases}
f_i 
= 
\frac{1}{3}t 
+ 
( 
\alpha_{i+1}-\alpha_{i+2}-3\alpha_{i+3}-\alpha_{i+4} 
) 
t^{-1} + \cdots \\
f_{i+1} 
= 
\frac{1}{3}t 
+ 
( 
\alpha_{i+2}-\alpha_i-\alpha_{i+3}+\alpha_{i+4} 
) 
t^{-1} + \cdots \\
f_{i+2} 
= 
\frac{1}{3}t 
+ 
( 
\alpha_i-\alpha_{i+1}+\alpha_{i+3}+3\alpha_{i+4} 
) t^{-1}
+ \cdots \\
f_{i+3} 
= 
3\alpha_{i+3} 
t^{-1} + \cdots \\
f_{i+4} 
= 
-3\alpha_{i+4} t^{-1} + \cdots.
\end{cases}
\end{equation*}
Type C: \quad  $f_0, f_1, f_2, f_3, f_4$ all have a pole at $t=\infty,$ 
and 
$f_0, f_{1},f_{2},f_{3},f_{4}$ are uniquely expanded as follows:
\begin{equation*}
\begin{cases}
f_0 
= 
1/5t 
+ 
( 
3\alpha_1+\alpha_2-\alpha_3-3\alpha_4
) 
t^{-1}  
+ \cdots \\
f_1 
= 
1/5t 
+ 
( 
3\alpha_2+\alpha_3-\alpha_4-3\alpha_0 
) 
t^{-1} + \cdots \\
f_2 
= 
1/5
t 
+ 
( 
3\alpha_3+\alpha_4-\alpha_0-3\alpha_1 
) 
t^{-1} + \cdots \\
f_3 
= 
1/5t 
+ 
( 
3\alpha_4+\alpha_0-\alpha_1-3\alpha_2 
) 
t^{-1} + \cdots \\
f_4 
= 
1/5t 
+ 
( 
3\alpha_0+\alpha_1-\alpha_2-3\alpha_3 
) 
t^{-1} + \cdots. 
\end{cases}
\end{equation*}
}
\end{proposition}

From the uniqueness in Proposition \ref{prop:a4inf}, 
we have the following corollary. 

\begin{corollary}
\label{coro:a4odd}
{\it
Let $(f_j)_{0 \leq j \leq 4}$ be a meromorphic solution at $t=\infty$ of 
$A_4^{(1)}(\alpha_j)_{0\leq j \leq 4}$. 
Then, 
$ f_j \,(0\leq j \leq 4)$ are odd functions.
}
\end{corollary}

\begin{proof}
$A^{(1)}_4(\alpha_j)_{0\leq j \leq 4}$ 
is invariant under the transformation defined by 
\begin{equation*}
s_{-1}:
t 
\longrightarrow 
-t, \quad
f_j 
\longrightarrow
-f_j 
\quad
(0\leq j \leq 4).
\end{equation*}
Each of 
Type A, Type B, Type C on Proposition \ref{prop:a4inf} 
is also invariant under $s_{-1}$. 
Then 
$ f_j ( t ) = - f_j ( - t ) \,\,(0\leq j \leq 4)$, 
because the Laurent series of $f_j$ at $t=\infty$ on each of the types are unique. 
Therefore, $ f_j $ are odd functions. 
\end{proof}

\section{Meromorphic Solution at $t=c \in \mathbb{C}$}
In this section, 
we treat the meromorphic solutions at $t=c\in\mathbb{C}$ 
such that some of $f_j \,\,(0\leq j \leq 4)$ have a pole at $t=c.$  
The main results of this section were proved by Tahara \cite{Tahara}, 
who calculated the Laurent series of $f_j \,\,(0\leq j \leq 4)$ at $t=c.$

\begin{proposition}
\label{prop:a4c}
{\it
Suppose that 
for $A_4^{(1)}(\alpha_j)_{0\leq j \leq 4},$ 
there exists a meromorphic solution at $t=c\in\mathbb{C}$ 
such that  
some of $(f_j)_{0\leq j \leq 4}$ have a pole at $t=c \in \C.$ 
One of the following then occurs:
\newline
(1) \quad  for some $i=0, 1, 2, 3, 4,$ $f_i, f_{i+1}$ both have a pole at $t=c \in \C$ 
and $f_{i+2}, f_{i+3}, f_{i+4}$ are all holomorphic at $t=c,$ 
and 
$f_i, f_{i+1}, f_{i+2}, f_{i+3}, f_{i+4}$ are expanded as follows:
\begin{equation*}
\begin{cases}
f_i 
= 
(t-c)^{-1} 
+ 
\frac{c}{2} 
+ 
\left(
1+\frac{c^2}{12}-\frac{1}{3}\alpha_i-\frac{2}{3}\alpha_{i+1} 
-\frac{2}{3}\alpha_{i+3}
\right)
(t-c)  \\
\quad
+ 
\left( 
- 
\frac{1}{2} 
(q_{i+2,2}+q_{i+4,2})
+ 
\frac{c}{8}+\frac{c}{4}(\alpha_{i+2}+\alpha_{i+4}) 
\right)
(t-c)^2+\cdots   \\
f_{i+1} 
=
- 
(t-c)^{-1} 
+ 
\frac{c}{2} 
+ 
\left(
1-\frac{c^2}{12}-\frac{2}{3}\alpha_i-\frac{1}{3}\alpha_{i+1} 
-\frac{2}{3}\alpha_{i+3} 
\right) 
(t-c) \\
\quad 
+ 
\left( 
- 
\frac{1}{2} (q_{i+2,2}+q_{i+4,2}) 
- 
\frac{1}{8}c-\frac{c}{4} 
(\alpha_{i+2}+\alpha_{i+4}) 
\right)
(t-c)^2 \cdots  \\
f_{i+2} 
= 
-\alpha_{i+2}(t-c)+q_{i+2,2} (t-c)^2 + \cdots \\
f_{i+3} 
= 
\frac{\alpha_{i+3}}{3}(t-c) 
+ 
0(t-c)^2 + \cdots \\
f_{i+4} 
= 
-\alpha_{i+4}(t-c)+ q_{i+4,2} (t-c)^2 \cdots,   
\end{cases}
\end{equation*} 
where $q_{i+2,2}, q_{i+4,2}$ are arbitrary constants; 
\newline
(2) \quad 
for some $i=0,1,2,3,4,$ $f_i, f_{i+2}$ both have a pole at $t=c \in \mathbb{C}$ 
and $f_{i+1},f_{i+3}, f_{i+4}$ are all holomorphic at $t=c,$ 
and 
$f_i, f_{i+1}, f_{i+2}, f_{i+3}, f_{i+4}$ are expanded as follows:
\begin{equation*}
\begin{cases}
f_i 
= 
- 
(t-c)^{-1} 
+ 
\left( 
\frac{1}{2}c-q_{i+3,0} 
\right)  \\
\quad
+  
\left( 
\frac{1}{3} 
\left( 
2 
+ 
\alpha_{i+1}-\alpha_{i+2}-3\alpha_{i+3}-\alpha_{i+4}
\right)  
+ 
\frac{2}{3}q_{i+3,0} 
\left( 
c-q_{i+3,0}-2q_{i+4,0} 
\right) 
- 
\frac{1}{3} 
\left( 
\frac{1}{2}c-q_{i+3,0}
\right)^2 
\right) \\
\hspace{20mm}
\times(t-c) 
+ \cdots \\
f_{i+1} 
=
-\alpha_{i+1}(t-c) + \cdots   \\
f_{i +2} 
= 
(t-c)^{-1} 
+ 
\left(
\frac{1}{2}c-q_{i+4,0}
\right)  \\
\quad
+ 
\left( 
\frac{1}{3} 
(2-\alpha_i+\alpha_{i+1}-\alpha_{i+3}-3\alpha_{i+4}) 
- 
\frac{2}{3}q_{i+4,0} 
(c-2q_{i+3,0}-q_{i+4,0}) 
+ 
\frac{1}{3} 
(
\frac{1}{2}c-q_{i+4,0}
)^2 
\right) \\
\hspace{20mm}
\times(t-c) + \cdots  \\
f_{i+3} 
= 
q_{i+3,0} 
+ 
\big( 
q_{i+3,0} 
(-c+q_{i+3,0}+2 q_{i+4,0})+\alpha_{i+3} 
\big) (t-c) + \cdots \\
f_{i+4} 
= 
q_{i+4,0} 
+ 
\big( 
q_{i+4,0} 
(c-2q_{i+3,0}-q_{i+4,0})+\alpha_{i+ 4} 
\big) (t-c) + \cdots, 
\end{cases}
\end{equation*}
where $q_{i+3,0}, q_{i+4,0}$ are arbitrary constants;
\newline
(3) \quad for some $i=0,1,2,3,4,$ $f_{i+1}, f_{i+2}, f_{i+3}, f_{i+4}$ all have a pole at $t=c \in \C$ 
and $f_i$ is holomorphic at $t=c,$ 
and 
$f_i, f_{i+1}, f_{i+2}, f_{i+3}, f_{i+4}$ are expanded as follows: 
\begin{equation*}
\begin{cases} 
f_i = - \frac{\alpha_i}{3} (t-c) + \cdots   \\
f_{i+1} 
=
3(t-c)^{-1} 
+ 
\left(
\frac{c^2}{10}-\frac{2}{5}-\frac{3}{5}\alpha_i
+\frac{3}{5}\alpha_{i+2}+\frac{1}{5}\alpha_{i+3}
-\frac{1}{5}\alpha_{i+4}
\right) (t-c) + \cdots  \\
f_{i+2} 
=
(t-c)^{-1} 
+ 
\frac{c}{2}
+
\left(
\frac{c^2}{12}+\frac{2}{3}+\alpha_i+\frac{1}{3}\alpha_{i+1} 
-\frac{1}{3}\alpha_{i+3}+\frac{1}{3} \alpha_{i+4} 
\right) 
(t-c) + \cdots  \\
f_{i+3} 
= 
- 
(t-c)^{-1} 
+ \frac{ c }{ 2 } 
+ 
\left( 
-\frac{c^2}{12}+\frac{2}{3}+\alpha_i+\frac{1}{3}\alpha_{i+1} 
-\frac{1}{3}\alpha_{i+2}+\frac{1}{3}\alpha_{i+4}
\right) (t-c) + \cdots  \\
f_{i+4} 
= 
-3 
(t-c)^{-1} 
+ 
\left(
-\frac{c^2}{10}-\frac{2}{5}-\frac{3}{5}\alpha_i 
-\frac{1}{5}\alpha_{i+1}+\frac{1}{5}\alpha_{i+2} 
+\frac{3}{5} \alpha_{i+3} 
\right) (t-c) + \cdots.
\end{cases}
\end{equation*}
}
\end{proposition}

From Proposition \ref{prop:a4c}, we obtain the following corollary:

\begin{corollary}
\label{coro:res}
{\it
(1)\quad Suppose that $A^{(1)}_4(\alpha_j)_{0\leq j \leq 4}$ has a rational solution. 
For any $i=0,1,2,3,4,$ then $\Res_{t=\infty} f_i \in\mathbb{Z}.$
\newline
(2)
\quad Suppose that $A^{(1)}_4(\alpha_j)_{0\leq j \leq 4}$ has a rational solution such that 
for some $i=0,1,2,3,4,$ $f_i$ has a pole at $t=c\in\mathbb{C}\setminus\{0\}.$ 
$ f_i $ then also has a pole at $ t=- c,$ and 
$\Res_{t=c}  f_i = \Res_{ t=- c } f_i. $
\newline
(3)
\quad 
Suppose that $A^{(1)}_4(\alpha_j)_{0\leq j \leq 4}$ has a rational solution such that 
for some $i=0,1,2,3,4,$ $\Res_{t=\infty} f_i$ is an even integer. 
$f_i$ is then holomorphic at $t=0$ and  
$$
f_i = a_{i,1} t + \sum_{ j = 1 }^{ n_i } 
\left( \frac{ \varepsilon_{i,j} }{ t - c_{i,j} } + \frac{ \varepsilon_{i,j} }{ t + c_{i,j} } \right),
$$
where $ a_{i,1} = 0 , \pm 1 , \frac{ 1 }{ 3 } , \frac{ 1 }{ 5 } $ and $ \varepsilon_{i,j} = \pm 1 , \pm 3 $ 
and $c_{i,j} \neq 0$. 
\newline
(4)
\quad 
Suppose that $A^{(1)}_4(\alpha_j)_{0\leq j \leq 4}$ has a rational solution such that 
for some $i=0,1,2,3,4,$ $\Res_{t=\infty} f_i$ is an odd integer. 
$f_i$ then has a pole at $t=0$ and
$$
f_i = a_{i,1} t +  \frac{ \varepsilon_{i,0} }{ t } + \sum_{ j = 1 }^{ n_i } 
\left( \frac{ \varepsilon_{i,j} }{ t - c_{i,j} } + \frac{ \varepsilon_{i,j} }{ t + c_{i,j} } \right),
$$
where $ \varepsilon_{i,0}, \varepsilon_{i,j} = \pm 1 , \pm 3 $ and $c_{i,j} \neq 0$. 
}
\end{corollary}

\begin{proof}
Case (1) can be proved by applying the residue theorem to $f_i\,(0\leq i \leq 4).$ 
We treat case (2). For this purpose, 
let $ c \in \C \setminus \{ 0 \} $ be a pole of $ f_i $. 
Then 
it follows 
from Proposition \ref{prop:a4c} and Corollary \ref{coro:a4odd} that 
$ f_i $ has a pole at $t=c$ with the first order and is an odd function:
\begin{equation*}
f_i ( t ) =- f_i ( - t ).
\end{equation*}
Therefore, 
$ - c $ is also a pole of $ f_i $ and $ \Res_{t=c} f_i = \Res_{ - c } f_i $.
\par
We next deal with case (3). 
For this purpose, 
suppose that $f_i$ has a pole at $t=0$ 
and $ \pm c_1 , \pm c_2 , \cdots \pm c_{n_i} \in \C \setminus \{ 0 \} $ are poles of $ f_i .$ 
It then follows from the residue theorem that 
$$
-\Res_{t=\infty} f_i  = \Res_{t=0} f_i + 2 \sum_{ j = 1 }^{ n_i }  \Res_{t=c_j} f_i,
$$
which is a contradiction 
because $\Res_{t=0} f_i = \pm 1 \, \rm{or} \,  \pm 3. $ 
\par
We last treat case (4). 
For this purpose, 
suppose that $f_i $ is holomorphic at $t=0$ and 
$ \pm c_1 , \pm c_2 , \cdots \pm c_{n_i} \in \C \setminus \{ 0 \} $ are poles of $ f_i .$ 
It then follows from the residue theorem that 
$$
-\Res_{t=\infty} f_i = 2 \sum_{ j = 1 }^{ n_i }  \Res_{t=c_j} f_i, 
$$
which is a contradiction. 
\end{proof}

\section{The Laurent Series of the Auxiliary Function}
In this section, following Noumi and Yamada \cite{NoumiYamada-B}, 
we first introduce the Hamiltonian $H$ of $A_4^{(1)}(\alpha_i)_{0\leq i \leq 4}$ 
and its principal part $\hat{H}$. 
In this paper, 
we consider $\hat{H}$ as the auxiliary function 
in order to determine the rational solutions of $A_4^{(1)}(\alpha_i)_{0\leq i \leq 4}.$ 
\par
For this purpose, 
using Propositions \ref{prop:a4inf} and \ref{prop:a4c}, 
we calculate the residues of $\hat{H}$ at $t=\infty,c\in\mathbb{C}.$ 
\par
Noumi and Yamada \cite{NoumiYamada-B} defined 
the Hamiltonian 
$H$ of $A_4^{(1)}(\alpha_j)_{0\leq j \leq 4}$ 
by
\begin{align*}
\label{eqn:H}
H
&=
\hat{H}   \\
&
+\frac15 
\left(
2\alpha_1-\alpha_2+\alpha_3-2\alpha_4
\right)f_0
+\frac15
\left(
2\alpha_1+4\alpha_2+\alpha_3+3\alpha_4
\right)f_1 \\
&
-\frac15
\left(
3\alpha_1+\alpha_2-\alpha_3+2\alpha_4
\right)f_2
+\frac15
\left(
2\alpha_1-\alpha_2+\alpha_3+3\alpha_4
\right)f_3 \\
&
-\frac15
\left(
3\alpha_1+\alpha_2+4\alpha_3+2\alpha_4
\right)f_4.
\end{align*}
where 
$\hat{H}$ is called the principal part of $H$ and is defined by 
\begin{equation*}
\hat{H}
=
f_0f_1f_2+
f_1f_2f_3+
f_2f_3f_4+
f_3f_4f_0+
f_4f_0f_1.   
\end{equation*}

\subsection{The Laurent series of $\hat{H}$ at $t=\infty$}
In this subsection, 
using the meromorphic solutions at $t=\infty$ in Proposition \ref{prop:a4inf},  
we compute the residue of the principal part of the Hamiltonian, $\hat{H},$ at $t=\infty.$
\par
For this purpose, 
we first note that $\hat{H}$ has a pole of an order of at most three at $t=\infty,$ 
because Proposition \ref{prop:a4inf} shows that 
$f_i\,\,(0\leq i \leq 4)$ is holomorphic at $t=\infty$ or 
has a pole at $t=\infty$ with the first order. 
Since Corollary \ref{coro:a4odd} shows that $f_i \,\,(0\leq i \leq 4)$ are all odd functions, 
the Laurent series of $\hat{H}$ at $t=\infty$ are given by 
\begin{equation*}
\hat{H}
:=
h_{\infty,3} t^{3} +
h_{\infty,1} t + 
h_{\infty,-1}t^{-1}+
O(t^{-3}) 
\,\, 
\mathrm{at} \,\,
t=\infty.
\end{equation*}

\begin{proposition}
\label{prop:phinf}
{\it
Suppose that $A_4^{(1)}(\alpha_j)_{0\leq j \leq 4}$ has a meromorphic solution at $t=\infty.$
\newline
(1)\quad 
Assume that Type A (1) occurs: 
for some $i=0,1,2,3,4,$ 
$f_i$ has a pole at $t=\infty$ and 
$f_{i+1},f_{i+1}, f_{i+2}, f_{i+3}$ are all holomorphic at $t=\infty.$ Then, 
\begin{equation*}
h_{\infty,-1}
=
-\alpha_{i+1}\alpha_{i+2}-\alpha_{i+3}\alpha_{i+4}-\alpha_{i+4}\alpha_{i+1}.
\end{equation*}
(2)\quad 
Assume that Type A (2) occurs: 
for some $i=0,1,2,3,4,$ 
$f_i, f_{i+1},f_{i+3}$ all have a pole at $t=\infty$ and $f_{i+2},f_{i+4}$ are both holomorphic at $t=\infty.$ 
Then, 
\begin{equation*}
h_{\infty,-1}
=
-\alpha_{i+2}(\alpha_i+\alpha_{i+3})
+\alpha_{i+4}(\alpha_{i+1}+\alpha_{i+3})
+\alpha_{i+2}\alpha_{i+4}.
\end{equation*}
(3)\quad 
Assume that Type B occurs: 
for some $i=0,1,2,3,4,$ 
$f_i,f_{i+1},f_{i+2}$ all have a pole at $t=\infty$ and $f_{i+3},f_{i+4}$ are both  holomorphic at $t=\infty.$ Then,  
\begin{equation*}
h_{\infty,-1}
=
\frac13
\big\{
-
(\alpha_i-\alpha_{i+1}+\alpha_{i+3})^2
-
(\alpha_{i+2}-\alpha_i-\alpha_{i+3}+\alpha_{i+4})
(\alpha_{i+2}+\alpha_{i+4}-\alpha_{i+1})
-9\alpha_{i+3}\alpha_{i+4}
\big\}.
\end{equation*}
(4)\quad 
Assume that Type C occurs:  
$f_0,f_1,f_2,f_3,f_4$ all have a pole at $t=\infty$. Then,  
\begin{equation*}
h_{\infty,-1}
=
\frac15
(
-a_{\infty, -1}^2+a_{\infty,-1}e_{\infty,-1}-b_{\infty,-1}^2
-a_{\infty,-1}c_{\infty,-1}-c_{\infty,-1}^2
+c_{\infty,-1}d_{\infty,-1}+2d_{\infty,-1}e_{-\infty,1}
),
\end{equation*}
where 
\begin{align*}
&
a_{\infty,-1}=3\alpha_1+\alpha_2-\alpha_3-3\alpha_4, 
b_{\infty,-1}=3\alpha_2+\alpha_3-\alpha_4-3\alpha_0, 
c_{\infty,-1}=3\alpha_3+\alpha_4-\alpha_0-3\alpha_1,  \\
&
d_{\infty,-1}=3\alpha_4+\alpha_0-\alpha_1-3\alpha_2, 
e_{\infty,-1}=3\alpha_0+\alpha_1-\alpha_2-3\alpha_3.
\end{align*}
}
\end{proposition}

\subsection{The Laurent series of $\hat{H}$ at $t=c \in\mathbb{C}$}

In this subsection, 
using the meromorphic solutions at $t=c\in\mathbb{C}$ in Proposition \ref{prop:a4c},  
we calculate the residue of the principal part of the Hamiltonian, $\hat{H},$ at $t=c.$

\begin{proposition}
\label{prop:H}
{\it
Suppose that $A_4^{(1)}(\alpha_j)_{0\leq j \leq 4}$ has a meromorphic solution at $t=c\in\mathbb{C}.$ 
\newline
(1)\quad 
Assume that for some $i=0,1,2,3,4,$ 
$f_i , f_{ i + 1 }$ both have a pole at $t=c$ 
and $f_{i+2},f_{i+3},f_{i+4}$ are all holomorphic at $t=c.$ Then, 
\begin{equation*} 
\Res_{t=c} \hat{H} =\alpha_{ i + 2 } + \alpha_{ i + 4 }.
\end{equation*}
(2)\quad 
Assume that for some $i=0,1,2,3,4,$ $f_i , f_{ i + 2 }$ both have a pole at $t=c$ 
and $f_{i+1},f_{i+3},f_{i+4}$ are all holomorphic at $t=c.$ Then, 
\begin{equation*} 
\Res_{t=c} \hat{H} =
\alpha_{ i + 1 }; 
\end{equation*}
(3) \quad 
Assume that for some $i=0,1,2,3,4,$ $f_{ i + 1 }, f_{ i + 2 }, f_{ i + 3 }, f_{ i + 4 }$ all 
have a pole at $t=c$ and $f_i$ is holomorphic at $t=c.$ Then, 
\begin{equation*} 
\Res_{t=c} \hat{H} =
\alpha_{ i + 1 } + \alpha_{ i + 4 }.
\end{equation*}
}
\end{proposition}

\subsection{Rational solutions and the residue calculus of $\hat{H}$}

\begin{proposition}
\label{prop:fund}
{\it
Suppose that 
$A^{(1)}_4(\alpha_j)_{0\leq j \leq 4}$ 
has a rational solution and the parameters satisfy 
$0\leq \alpha_j \leq 1 \,\,(0\leq j \leq 4).$ 
Then, 
\begin{equation*}
h_{\infty,-1} \geq 0.
\end{equation*}
}
\end{proposition}

\begin{proof}
Let $c_1, \ldots, c_k \in \mathbb{C}$ be  
the poles of $(f_j)_{0\leq j \leq 4}$. 
Since $0 \leq \alpha_j \leq 1 \,\,(0\leq j \leq 4),$ 
it follows 
from Proposition \ref{prop:H} that 
\begin{equation*}
\Res_{t=c_l} \hat{H} \geq 0 \,\,(1 \leq l \leq k).
\end{equation*} 
Therefore, it follows from the residue theorem that 
\begin{equation*}
h_{\infty,-1}
=
-\Res_{t=\infty} \hat{H}
=
\sum_{l=1}^k 
\Res_{t=c_l}\hat{H}
\geq 0.
\end{equation*}
\end{proof}

\section{Necessary Conditions}
In this section, 
we obtain the necessary condition for 
$A^{(1)}_4(\alpha_i)_{0\leq i \leq 4}$ to have a rational solution. 
For this purpose, 
we use Corollary \ref{coro:res}.

\subsection{Necessary conditions for Type A}

\begin{theorem}
\label{thm:a4par-A}
{\it
Suppose that $A^{(1)}_4(\alpha_j)_{0\leq j \leq 4}$ has a rational solution of Type A. 
The parameters then satisfy 
$ 
\alpha_i \in \Z \,\,(0\leq i \leq 4).
$
}
\end{theorem}

\begin{proof}
Proposition \ref{prop:a4c} implies that 
$\Res_{t=c} f_i =\pm1, \pm 3 \, ( 0 \leq i \leq 4 )$ for $c \in \mathbb{C}$. 
Therefore, it follows from the residue theorem 
that $\Res_{t=\infty} f_i \in \Z \,\,(0\leq i \leq 4).$
\par
From Proposition \ref{prop:a4inf}, 
it follows that either Type A (1) or Type A (2) occurs. 
If Type A (1) happens, 
from the residue theorem and  Proposition \ref{prop:a4inf}, 
we find that $\alpha_{i+1}, \alpha_{i+2}, \alpha_{i+3}, \alpha_{i+4} \in \Z,$ 
which proves that $\alpha_i \in \Z$ because $\sum_{k=0}^4 \alpha_k=1.$ 
If Type A (2) occurs,  
we can show that $\alpha_j \in \Z \,\,(0\leq j \leq 4)$ in the same way as Type A (1).

\end{proof}

\subsection{Necessary condition for Type B}

\begin{theorem}
\label{thm:a4par-B}
{\it
Suppose that $A^{(1)}_4(\alpha_j)_{0\leq j \leq 4}$ has a rational solution of Type B. 
The parameters then satisfy 
$$ 
(\alpha_i, \alpha_{i+1}, \alpha_{i+2}, \alpha_{i+3}, \alpha_{i+4})
\equiv
\left( 
\frac{n_1}{3}-\frac{n_3}{3}, \, \frac{n_1}{3}, \,  
\frac{n_1}{3}+\frac{n_4}{3}, \, \frac{n_3}{3}, \,  
- \frac{n_4}{3}
\right) \quad \textrm{mod} \, \,
\mathbb{Z},  \,\,n_1, n_3, n_4=0, 1, 2, 
$$
for some $i=0,1,2,3,4.$ 
}
\end{theorem}

\begin{proof}
Proposition \ref{prop:a4c} implies that 
$\Res_{t=c} f_i =\pm1, \pm 3 \, ( 0 \leq i \leq 4 )$ for $c \in \mathbb{C}$. 
Therefore, it follows from the residue theorem 
that $\Res_{t=\infty} f_i \in \Z \,\,(0\leq i \leq 4).$
\par
It then follows 
from Proposition \ref{prop:a4inf} that 
$\Res_{t=\infty} f_{i+3} \in\mathbb{Z}$ and $\Res_{t=\infty} f_{i+4} \in \Z,$ 
which imply that 
$$
\alpha_{i+3}=n_3/3 \,\,\mathrm{and} \,\,
\alpha_{i+4} = -n_4/3, \, n_3 , n_4 \in \mathbb{Z}.
$$ 
Furthermore, 
Proposition \ref{prop:a4inf} shows that $\Res_{t=\infty}f_{i+1}\in\mathbb{Z}$ and $\Res_{t=\infty}f_{i+2} \in \Z,$ 
which imply that 
\begin{equation*}
\alpha_{i+2}-\alpha_i-\frac{n_3}{3}-\frac{n_4}{3} = m_1\in \mathbb{Z},  \,\,
\alpha_i-\alpha_{i+1}+\frac{n_3}{3}-n_4           = m_2 \in \mathbb{Z},
\end{equation*}
respectively. 
By solving this system of equations of 
$\alpha_{i}, \alpha_{i+2}$, 
we obtain
\begin{alignat*}{4}
&\alpha_i      &  &= & \,\, &\alpha_{i+1}-\frac{n_3}{3}+m_2+n_4  \\
&\alpha_{i+1}  &  &= & \,\, &\alpha_{i+1}  \\
&\alpha_{i+2}  &  &= & \,\, &\alpha_{i+1}+\frac{n_4}{3}+m_1+m_2+n_4.
\end{alignat*}
Since
$
\alpha_{i+3}=n_3/3, 
\alpha_{i+4}=-n_4/3 
$ 
and 
$\sum_{j=0}^{4} \alpha_j = 1,$ 
it follows that 
$\alpha_{i+1}=n_1/3$ for some integer $n_1 \in \mathbb{Z},$ 
which implies that 
$$
(\alpha_i, \alpha_{i+1}, \alpha_{i+2}, \alpha_{i+3}, \alpha_{i+4}) 
\equiv 
\left(
\frac{n_1}{3}-\frac{n_3}{3}, 
\frac{n_1}{3}, 
\frac{n_1}{3}+\frac{n_4}{3}, 
\frac{n_3}{3}, 
-\frac{n_4}{3} 
\right) \,\, \mathrm{mod}\,\,\Z.
$$

\end{proof}

\subsection{Necessary condition for Type C}

\begin{theorem}
\label{thm:a4par-C}
{\it
Suppose that $A^{(1)}_4(\alpha_j)_{0\leq j \leq 4}$ has a rational solution of Type C. 
The parameters then satisfy 
\begin{align*}
(\alpha_i, \alpha_{i+1}, \alpha_{i+2}, \alpha_{i+3}, \alpha_{i+4})
&\equiv
\left( 
\frac{n_1}{5}+\frac{2n_2}{5}+\frac{3n_3}{5}, \, 
\frac{n_1}{5}+\frac{2n_2}{5}+\frac{n_3}{5}, \, \frac{n_1}{5}, \, 
\frac{n_1}{5}+\frac{n_2}{5}, \, \frac{n_1}{5}+\frac{n_3}{5} 
\right) 
\,
\textrm{mod} \, \,\Z,   \\
&\hspace{50mm}n_1, n_2, n_3=0, 1, 2, 3, 4,
\end{align*} 
for some $i=0,1,2,3,3,4.$
}
\end{theorem}

\begin{proof}
Proposition \ref{prop:a4c} implies that 
$\Res_{t=c} f_i =\pm1, \pm 3 \, ( 0 \leq i \leq 4 )$ for $c \in \mathbb{C}$. 
Therefore, it follows from the residue theorem 
that $\Res_{t=\infty} f_i \in \Z \,\,(0\leq i \leq 4).$
\par
Then, 
it follows 
from Proposition \ref{prop:a4inf} that 
\begin{equation*}
(*)
\begin{cases}
3\alpha_1+\alpha_2-\alpha_3-3\alpha_4 &= m_0 \in \mathbb{Z}, \\
3\alpha_2+\alpha_3-\alpha_4-3\alpha_0 &= m_1 \in \mathbb{Z},\\
3\alpha_3+\alpha_4-\alpha_0-3\alpha_1 &= m_2 \in \mathbb{Z}, \\
3\alpha_4+\alpha_0-\alpha_1-3\alpha_2 &= m_3 \in \mathbb{Z}, \\
3\alpha_0+\alpha_1-\alpha_2-3\alpha_3 &= m_4 \in \mathbb{Z}.
\end{cases}
\end{equation*}
Solving this system of equations of $\alpha_j \,(0 \leq j \leq 4),$ we obtain 
\begin{align*}
\alpha_0 &= 
\alpha_3-\frac{3}{5}m_0-\frac{2}{5}m_1-\frac{2}{5}m_2-\frac{1}{5}m_3  \\
\alpha_1 
&= 
\alpha_3+\frac{1}{5}m_1-\frac{2}{5}m_2+\frac{1}{5}m_3  \\
\alpha_2 
&= 
\alpha_3-\frac{4}{5}m_0-\frac{3}{5}m_2-\frac{3}{5} m_3  \\
\alpha_3 &= 
\alpha_3  \\
\alpha_4 &=
\alpha_3-\frac{3}{5}m_0+\frac{1}{5}m_1-\frac{3}{5} m_2.
\end{align*}
Since $ \sum_{ i = 0 }^{ 4 } \alpha_i= 1 $, 
it follows that 
$
\alpha_j = n_j/5 \, n_j \in \mathbb{Z} \, (0 \leq j \leq 4).
$
Substituting $\alpha_j=n_j/5$ in $(*),$ 
we get 
\begin{equation*}
(**)
\begin{cases}
3n_1+n_2-n_3-3n_4 &\equiv 0 \,\, \textrm{mod} \,\, 5  \\
3n_2+n_3-n_4-3n_0 &\equiv 0 \,\, \textrm{mod} \,\, 5  \\
3n_3+n_4-n_0-3n_1 &\equiv 0 \,\, \textrm{mod} \,\, 5  \\
3n_4+n_0-n_1-3n_2 &\equiv 0 \,\, \textrm{mod} \,\, 5  \\
3n_0+n_1-n_2-3n_3 &\equiv 0 \,\, \textrm{mod} \,\, 5.
\end{cases}
\end{equation*}
By solving $(**)$ in the field $\mathbb{Z} / 5 \mathbb{Z}$, 
we obtain
\begin{alignat*}{7}
&n_0  & \,\, &\equiv  & \,\,  &l_1+2l_2+3l_3  & \,\, &\textrm{mod} \,\, 5,   \\
&n_1  & \,\, &\equiv  & \,\,  &l_1+2l_2+l_3   & \,\, &\textrm{mod} \,\, 5,  \\
&n_2  & \,\, &\equiv  & \,\,  &l_1            & \,\, &\textrm{mod} \,\, 5,   \\
&n_3  & \,\, &\equiv  & \,\,  &l_1+l_2        & \,\, &\textrm{mod} \,\, 5,   \\
&n_4  & \,\, &\equiv  & \,\,  &l_1+l_3        & \,\, &\textrm{mod} \,\, 5,
\end{alignat*}
where $l_1, l_2, l_3=0,1,2,3,4.$ 
\end{proof}

\section{Reduction of the Parameters}
In the previous section, 
we obtained the necessary conditions for $A_4^{(1)}(\alpha_j)_{0\leq j \leq 4}$ 
to have a rational solution, 
which are expressed by the parameters, $\alpha_j \,\,(0\leq j \leq 4).$ 
\par
In this section, 
we prove that 
by some B\"acklund transformations, 
the parameters can be transformed so that $0\leq \alpha_i \leq 1\,\,(0\leq i \leq 4).$

\subsection{Shift operators}

Noumi and Yamada \cite{NoumiYamada-A} defined the shift operators in the following way:
\begin{proposition}
\label{prop:shift}
{\it
For any $i=0,1,2,3,4,$
$T_i$ 
denote the shift operators 
which 
are expressed by 
$$ 
T_1 = \pi s_4 s_3 s_2 s_1, \, 
T_2 = s_1 \pi s_4 s_3 s_2, \, 
T_3 = s_2 s_1 \pi s_4 s_3, \, 
T_4 = s_3 s_2 s_1 \pi s_4, \, 
T_0 = s_4 s_3 s_2 s_1 \pi. 
$$
Then, 
$$
T_i (\alpha_{i-1}) 
= 
\alpha_{i-1}+1, \, 
T_i(\alpha_i) 
=\alpha_i-1, \, 
T_i(\alpha_j)=\alpha_j \, (j \neq i - 1, i).
$$
}
\end{proposition}

We can easily check that 
Type A, Type B and Type C are invariant under 
all the B\"acklund transformations. 

\begin{proposition}
\label{prop:relation-s_i-sol}
{\it
(1)\quad Suppose that 
$(f_k)_{0\leq k \leq 4}$ is a meromorphic solution of 
Type A  at $t=\infty.$ 
For any $j=0,1,2,3,4,$ 
$s_j(f_k)_{0\leq k \leq 4}$ and $\pi(f_k)_{0\leq k\leq 4}$ are then all 
meromorphic solutions of Type A at $t=\infty. $ 
\newline 
(2)\quad Suppose that 
$(f_k)_{0\leq k \leq 4}$ is a meromorphic solution of Type B at $t=\infty.$  
For any $j=0,1,2,3,4,$ 
$s_j (f_k)_{0\leq k \leq 4}$ and $\pi(f_k)_{0\leq k\leq 4}$ are then all meromorphic solutions of Type B at $t=\infty.$  
\newline 
(3)\quad Suppose that 
$(f_k)_{0\leq k \leq 4}$ is a meromorphic solution of Type C at $t=\infty.$  
For any $j=0,1,2,3,4,$ 
$s_j (f_k)_{0\leq k \leq 4}$ and $\pi(f_k)_{0\leq k\leq 4}$ are then all meromorphic solutions  of Type C at $t=\infty$. 
}
\end{proposition}

\begin{proof}
It can be easily checked.

\end{proof}

Using Propositions \ref{prop:shift} and \ref{prop:relation-s_i-sol}, 
we prove that 
for a rational solution of Types A, B, C of $A^{(1)}_4(\alpha_j)_{0\leq j \leq 4},$ 
by some B\"acklund transformations, 
the parameters can be transformed so that 
$0\leq \alpha_j \leq 1 \,\,(0\leq j \leq 4).$

\subsection{The reduction of the parameters for Type A}

\begin{proposition}
\label{prop:reduction-typeA}
{\it
Suppose that $A_4^{(1)}(\alpha_j)_{\leq j \leq 4}$ has a rational solution Type A. 
By some B\"acklund transformations, 
the parameters can then be transformed so that 
$(\alpha_0,\alpha_1,\alpha_2,\alpha_3,\alpha_4)=(1, 0, 0, 0, 0).$
}
\end{proposition}

\begin{proof}
Proposition \ref{thm:a4par-A} shows that $\alpha_j =n_j \in\mathbb{Z}, \,(0\leq j \leq 4).$ 
We can then inductively prove that 
the parameters, $(\alpha_0,\alpha_1,\alpha_2,\alpha_3,\alpha_4)=(n_0,n_1,n_2,n_3,n_4), $  
can be transformed to $(1, 0, 0, 0, 0).$ 
For this purpose, 
we use induction with respect to the number of nonzero parameters. 
\par
If the number of nonzero parameters is one,  
by $\pi$, the parameters can be transformed to $(1, 0, 0, 0, 0).$
\par
If the number of nonzero parameters is two, 
either of the following occurs:
\newline
(1)\quad for some $i=0,1,2,3,4,$ $(\alpha_{i}, \alpha_{i+1}, \alpha_{i+2}, \alpha_{i+3}, \alpha_{i+4})=(n_i, n_{i+1}, 0, 0, 0), $ 
\newline
(2)\quad for some $i=0,1,2,3,4,$ $(\alpha_{i}, \alpha_{i+1}, \alpha_{i+2}, \alpha_{i+3}, \alpha_{i+4})=(n_i, 0, n_{i+2}, 0, 0).$  
\par
When case (1) occurs, by $\pi$ and $T_1,$ 
the parameters can be transformed to $(1, 0, 0, 0, 0).$ 
When case (2) happens, by $\pi$ and $T_2,$ 
the parameters can be transformed to $ (1-n, n, 0, 0, 0), \,n\in\mathbb{Z}.$ 
By $T_1,$ the parameters can then be transformed to $(1, 0, 0, 0, 0).$
\par
If the number of nonzero parameters is three, 
either of the following occurs:
\newline
(3)\quad for some $i=0,1,2,3,4,$ $(\alpha_{i}, \alpha_{i+1}, \alpha_{i+2}, \alpha_{i+3}, \alpha_{i+4})=(n_i, n_{i+1}, n_{i+2}, 0, 0), $ 
\newline
(4)\quad for some $i=0,1,2,3,4,$ $(\alpha_{i}, \alpha_{i+1}, \alpha_{i+2}, \alpha_{i+3}, \alpha_{i+4})=(n_i, n_{i+1}, 0, n_{i+3}, 0).$  
\par
When case (3) occurs, by $\pi$ and $T_2,$ 
the parameters can be transformed to $(1-n,n,0,0,0) \,\, n\in\mathbb{Z}.$ Then, from the above discussions, 
the parameters can be transformed to $(1,0,0,0,0).$ 
\par
When case (4) occurs, by $T_3,$ the parameters can be transformed to $(n_0, n_1,n_2, 0, 0, 0).$ 
From the above discussions, 
the parameters can then be transformed to $(1,0,0,0,0).$ 
\par
If the number of nonzero parameters is four, then, 
by $\pi,$ the parameters can be transformed to 
$(\alpha_0,\alpha_1,\alpha_2,\alpha_3,\alpha_4)=(n_0,n_1,n_2,n_3, 0).$ 
Thus, 
by $T_3,$ 
the parameters can be transformed to $(n_0,n_1,n_2+n_3,0, 0).$ Then, from the above discussions, 
the parameters can be transformed to $(1,0,0,0,0).$ 
\par
If the number of nonzero parameters is five, then, 
by $T_4,$ the parameters can be transformed to $(n_0,n_1,n_2,n_3+n_4, 0).$ Therefore, from the above discussions, 
the parameters can be transformed to $(1,0,0,0,0).$

\end{proof}

\subsection{Reduction of the parameters for Type B}

\begin{proposition}
\label{prop:reduction-typeB}
{\it
Suppose that $A_4^{(1)}(\alpha_j)_{\leq j \leq 4}$ has a rational solution of Type B. 
By some B\"acklund transformations, 
the parameters can then be transformed so that 
one of the following occurs:
\begin{align*} 
(\alpha_0,\alpha_1,\alpha_2,\alpha_3,\alpha_4) =&
(1/3, 1/3, 1/3, 0, 0), 
(2/3, 0, 0, 1/3, 0), 
(1/3, 0, 0, 2/3, 0), \\
&(0, 1/3, 0, 1/3, 1/3), 
(1,0,0,0,0).
\end{align*}
\par
Especially, 
the parameters are transformed into 
$(1/3, 1/3, 1/3, 0, 0)$  
if in Theorem \ref{thm:a4par-B}, 
\begin{equation*}
(n_1, n_3, n_4)
= 
(\pm1,0,0), (\pm1,0,\pm1), (\pm1, \pm1, 0), \pm(0,1,-1),
\end{equation*}
that is, if 
for some $ i = 0, 1, \ldots 4$,
$$
\quad  
(\alpha_i, \alpha_{i+1}, \alpha_{i+2}, \alpha_{i+3}, \alpha_{i+4}) \equiv
\left\{
\begin{matrix}
\pm 1/3(1, 1, 1, 0, 0) \quad {\rm mod}   \, \Z \\
\pm 1/3(1, -1, -1, 1, 0)\quad {\rm mod}   \, \Z. 
\end{matrix}
\right.
$$
Otherwise, the parameters are transformed to the others. 
}
\end{proposition}

\begin{proof}
By Theorem \ref{thm:a4par-B} and $\pi,$
we can transform the parameters  
$$
(\alpha_0,\alpha_1,\alpha_2,\alpha_3,\alpha_4)
=
\left(
\frac{n_1}{3}-\frac{n_3}{3}, 
\frac{n_1}{3}, 
\frac{n_1}{3}+\frac{n_4}{3}, 
\frac{n_3}{3}, 
-\frac{n_4}{3}
\right)\,\, \textrm{mod} \, \Z, \, 
n_1,n_3,n_4=0,1,2.
$$ 
We have to consider $3^3=27$ cases. 
We now treat the following five cases. 
The other twenty two cases can be proved in the same way.
\par
When 
$n_1=n_3=n_4=0$, 
it follows from Proposition \ref{prop:reduction-typeA} that 
$$
(\alpha_0,\alpha_1,\alpha_2,\alpha_3,\alpha_4) 
\longrightarrow 
(1, 0, 0, 0, 0).
$$
\par
When $n_1=1, n_3=0, n_4=2,$ by $\pi$, 
we get
$$
(1/3, 1/3, 0, 0, 1/3) 
\longrightarrow 
(1/3, 1/3, 1/3, 0, 0).              
$$
\par
When $n_1 =1=n_3=n_4=1$, by $s_0 \circ s_4$, we obtain 
$$
(0,1/3, 2/3,1/3,-1/3)
\longrightarrow
(1/3,0,2/3,0,0).
$$
\par
When $n_1=1, n_3=n_4=2$, by $s_0$, 
we have 
$$
(-1/3, 1/3, 0, 2/3, 1/3)
\longrightarrow
(1/3, 0, 0, 2/3, 0).
$$
\par
When $n_1=1=n_3=1, n_4=2$, by $\pi$, we get 
$$
(0, 1/3, 0, 1/3, 1/3)
\longrightarrow
(1/3, 0, 1/3, 1/3,0).
$$

\end{proof}

\subsection{Reduction of the parameters for Type C}

\begin{proposition}
\label{prop:reduction-typeC}
{\it
Suppose that 
$A_4^{(1)}(\alpha_j)_{\leq j \leq 4}$ has a rational solution Type C. 
Then, by some B\"acklund transformations, 
the parameters can be transformed so that 
one of the following occurs:
\begin{align*}
(\alpha_0,\alpha_1,\alpha_2,\alpha_3,\alpha_4)=&(1/5, 1/5, 1/5, 1/5, 1/5), 
(1, 0, 0, 0, 0), (3/5, 0, 1/5, 1/5, 0), \\
&(0, 2/5, 2/5, 0), 
(1/5, 2/5, 0, 0, 2/5),
(3/5, 1/5, 0, 0, 1/5). 
\end{align*}
Especially, 
the parameters in Theorem \ref{thm:a4par-C} are 
transformed into 
$(1/5, 1/5, 1/5, 1/5, 1/5)$ 
if 
\begin{align*}
(n_1, n_2, n_3) 
=&
(1, 0, 0), (1, 2, 2), (1, 0, 1), 
(1, 2, 3), (1, 1, 0), (2, 0, 0), \\
&
(2, 4, 4), (2, 0, 2), (2, 4, 1), 
(2, 2, 0), (2, 2, 1), (3, 0, 0), \\
& 
(3, 1, 1), (3, 3, 4), (3, 3, 0), 
(3, 0, 3), (3, 1, 4), (4, 0, 0),  \\
& 
(4, 3, 3), (4, 4, 0), (4, 4, 2), 
(4, 3, 2), (4, 0, 4),
\end{align*}
that is, if 
for some $i = 0, 1, \ldots, 4$,
$$
(\alpha_i, \alpha_{i+1}, \alpha_{i+2}, \alpha_{i+3}, \alpha_{i+4}) \equiv
\left\{
\begin{matrix}
  j/5(1, 1, 1, 1, 1) \quad {\rm mod} \, \Z    \\
  j/5(1, 2, 1, 3, 3) \quad {\rm mod} \, \Z,   
\end{matrix}
\right.
$$
with some $ j=1, 2, 3, 4.$ 
}
\end{proposition}

\begin{proof}
By Theorem \ref{thm:a4par-C} and $\pi,$
we can transform the parameters to 
\begin{align*}
(\alpha_0,\alpha_1,\alpha_2,\alpha_3,\alpha_4)
&= 
\left( 
\frac{n_1}{5}+\frac{2n_2}{5}+\frac{3n_3}{5}, 
\frac{n_1}{5}+\frac{2n_2}{5}+\frac{n_3}{5}, 
\frac{n_1}{5}, 
\frac{n_1}{5}+\frac{n_2}{5}, 
\frac{n_1}{5}+\frac{n_3}{5} 
\right)\, \textrm{mod}\,\Z \,  \\
&\hspace{30mm}n_1,n_2,n_3=0,1,2,3,4. 
\end{align*} 
We have to consider $5^3=125$ cases. 
We now treat the following six cases. 
The other 119 (=125$-$6) cases can be proved in the same way.
\par
When $n_1=n_2=n_3=0,$ by some shift operators, we get
$$
(\alpha_i)_{0\leq i \leq 4}
\longrightarrow
(1,0,0,0,0).
$$
\par
When $n_1=1, n_2=n_3=0,$ by some shift operators, we obtain 
$$
(\alpha_i)_{0\leq i \leq 4}
\longrightarrow
(
1/5,
1/5,
1/5,
1/5,
1/5
).
$$
\par
When $n_1=0, n_2=2, n_3=0,$ by  $\pi^{-1}\circ s_4 \circ s_0$, 
we have
$$ 
(-1/5, 4/5, 0, 2/5, 0) 
\longrightarrow
(3/5, 0, 1/5, 1/5, 0).
$$
\par
When $n_1=0, n_2=1, n_3=0,$ by some shift operators, we get 
$$
(\alpha_i)_{0\leq i \leq 4}
\longrightarrow
(2/5, 2/5, 0, 1/5, 0).
$$
\par
When $n_1=n_2=0, n_3=2,$ by some shift operators, we obtain 
$$
(\alpha_i)_{0\leq i \leq 4}
\longrightarrow
(1/5, 2/5, 0, 0, 2/5).
$$ 
\par
When $n_1=n_2=0, n_3=1,$ by some shift operators, we have
$$
(\alpha_i)_{0\leq i \leq 4}
\longrightarrow
( 3/5 , 1/5 , 0 , 0 , 1/5 ).
$$
\end{proof}

\section{Rational Solutions for Some Parameters}
In this section, 
we determine the rational solutions of $A_4^{(1)}(\alpha_j)_{0\leq j \leq 4}$ 
such that the parameters satisfy $0\leq \alpha_j\leq 1\,\,(0\leq j \leq 4).$

\subsection{Rational Solutions of $A^{(1)}_4(1,0,0,0,0)$}

\begin{table}[h]
\caption
{
the residues of $ \hat{H} $ at $t=c \in \C$ 
for $ A_4^{(1)}( 1,0,0, 0,0 ) $
}
\begin{center}
\begin{tabular}{|c|c|c|c|c|c|}
\hline
$i$ & 0 & 1 & 2 & 3 & 4  \\
\hline
$ \alpha_{ i + 2 } + \alpha_{ i + 4 } $ & 0 & $ 1$ & 
$ 0$ & $ 1$ & $ 0$ \\
\hline
$ \alpha_{ i + 1 } $ & $ 0$ & $ 0$ & $ 0$ & 0 & $1$ \\
\hline
$ \alpha_{ i + 1 } + \alpha_{ i + 4 } $ & $ 0$ & $1$ & 
$ 0$ & $ 0$ & $ 1$ \\
\hline
\end{tabular}
\end{center}
\end{table}

Using Table 1, 
we determine the rational solution of Type A of 
$A^{(1)}_4(1,0,0,0,0).$

\begin{proposition}
\label{prop:Apoint}
{\it
$A^{(1)}_4(1,0,0,0,0)$ 
has a rational solution 
of Type A and 
\begin{equation*}
(f_0,f_1,f_2,f_3,f_4)
=
(t,0,0,0,0),
\end{equation*}
and it is unique. 
}
\end{proposition}

\begin{proof}
By direct calculation, 
we find that 
$A^{(1)}_4(1,0,0,0,0)$ 
has a rational solution such that 
$(f_0,f_1,f_2,f_3,f_4)=(t,0,0,0,0). $ 
Furthermore, 
from Proposition \ref{prop:a4inf}, 
it follows that if $f_0$ has a pole at $t=\infty$ and $f_1,f_2,f_3,f_4$ are all holomorphic at $t=\infty,$ 
the rational solution is unique. 
We then have only to prove that $A^{(1)}_4(1,0,0,0,0)$ has no more rational solutions of Type A. 
For this purpose, 
we consider the other cases:
\newline
Type A (1)\quad for some $i=1,2,3,4,$ $f_i$ has a pole at $t=\infty$ and $f_{i+1},f_{i+2},f_{i+3},f_{i+4}$ are all holomorphic at $t=\infty,$
\newline
Type A (2)\quad for some $i=0,1,2,3,4,$ $f_i, f_{i+1},f_{i+3}$ all have a pole at $t=\infty$ and $f_{i+2},f_{i+4}$ are both holomorphic at $t=\infty.$ 
\par
From Proposition \ref{prop:phinf}, 
we see that 
for a rational solution of Type A of 
$A^{(1)}_4(1,0,0,0,0),$ 
$h_{\infty,-1}=0.$ 
Furthermore, 
from Proposition \ref{prop:H} and Table 1, 
we find that 
the residue of $\hat{H}$ at $t=c \in \C$ 
is nonnegative. 
It then follows from the residue theorem that for a rational solution of $A^{(1)}_4(1,0,0,0,0),$ 
$\Res_{t=c} \hat{H} =0.$
Therefore, 
Table 1 implies that 
\begin{align*}
&
(f_0,f_1), \,\, (f_2,f_3), \,\,  (f_4,f_0) \\
&
(f_0,f_2), \,\,(f_1,f_3),\,\,(f_2,f_4),\,\,(f_3,f_0) \\
&
(f_1,f_2,f_3,f_4),\,\,(f_3,f_4,f_0,f_1),\,\,(f_4,f_0,f_1,f_2)
\end{align*}
can have a pole at $t=c \in \C.$
\par
We treat the case where 
$f_1$ has a pole at $t=\infty$ and show a contradiction. 
From the other three cases of Type A (1), 
we can prove the contradiction in the same way. 
Proposition \ref{prop:a4inf} implies that 
\begin{equation*}
-\Res_{t=\infty} f_1=1, 
f_2=f_3=f_4 \equiv 0, 
-\Res_{t=\infty} f_0=-1.
\end{equation*}
By considering $ f_2=f_3=f_4 \equiv 0,$ 
we then find from Corollary \ref{coro:res} 
that 
$f_0,f_1$ both have a pole at $t=0.$ 
On the other hand, 
it follows from Proposition \ref{prop:a4c} that 
$\Res_{t=0} f_0=1, \, \Res_{t=0} f_1=-1,$ 
which contradicts the residue theorem.
\par
We treat Type A (2). 
When $f_0,f_1,f_3$ all have a pole at $t=\infty,$ 
Proposition \ref{prop:a4inf} shows that 
\begin{equation*}
-\Res_{t=\infty} f_0=0, \,\,
-\Res_{t=\infty} f_1=1, \,\,
f_2 \equiv 0, \,\,
-\Res_{t=\infty} f_3=-1, \,\,
f_4 \equiv 0.
\end{equation*}
Since $f_2 =f_4 \equiv 0,$ 
\begin{equation*}
(f_0,f_1), \,\,
(f_1,f_3),\,\,(f_3,f_0) 
\end{equation*}
can have a pole in $\mathbb{C}.$ 
If $(f_0,f_1)$ or $(f_3,f_0)$ has a pole at $t=c \in \C,$ 
it follows 
from Proposition \ref{prop:a4c} that $\Res_{t=c} f_0=1,$ 
which contradicts the residue theorem. 
If $(f_1,f_3)$ has a pole at $t=c \in \mathbb{C},$ 
it follows from 
Proposition \ref{prop:a4c} that $\Res_{t=c} f_1=-1,\,\,\Res_{t=c} f_3=1,$ 
which contradicts the residue theorem. 
\par
When $f_1,f_2,f_4$ all have a pole at $t=\infty,$ 
it follows 
from Proposition \ref{prop:a4inf} that 
\begin{equation*}
-\Res_{t=\infty} f_1=1, \,\,
-\Res_{t=\infty} f_2=3, \,\,
f_3 \equiv 0, \,\,
-\Res_{t=\infty} f_4=-3, \,\,
-\Res_{t=\infty} f_0 =-1.
\end{equation*}
Therefore, 
\begin{equation*}
(f_0,f_1), \,\,(f_4,f_0) \,\,
(f_0,f_2), \,\,(f_2,f_4),\,\,
(f_4,f_0,f_1,f_2)
\end{equation*}
can have a pole in $\C$ 
because $f_3 \equiv 0.$ 
When 
$(f_4,f_0),$ $(f_2,f_4),$ 
$(f_4,f_0,f_1,f_2)$ 
has a pole at $t=c \in \mathbb{C},$ 
it follows 
from Proposition \ref{prop:a4c} that 
$
\Res_{t=c} f_4=1, \,\, 3,
$ 
which contradicts the residue theorem. 
Therefore, 
$f_4$ is holomorphic in $\mathbb{C},$ which is impossible.  
\par
When $f_2,f_3,f_0$ all have a pole at $t=\infty,$ 
it follows 
from Proposition \ref{prop:a4inf} that 
\begin{equation*}
-\Res_{t=\infty} f_2=1, \,\,
-\Res_{t=\infty} f_3=1, \,\,
f_4 \equiv 0, \,\,
-\Res_{t=\infty} f_0=-1, \,\,
-\Res_{t=\infty} f_1 =-1.
\end{equation*}
Corollary \ref{coro:res} then shows that 
$(f_0,f_1,f_2,f_3)$ has a pole at $t=0$ 
because $\Res_{t=\infty} f_j \,\,(0\leq j \leq 4)$ are odd integers. 
Proposition \ref{prop:H} implies that $\Res_{t=0} \hat{H}=1.$ 
Since $-\Res_{t=\infty} \hat{H}=h_{\infty,-1}=0$ and $\Res_{t=c} \hat{H}$ is nonnegative,  
this contradicts the residue theorem. 
\par
When $f_3,f_4,f_1$ all have a pole at $t=\infty,$ 
it follows 
from Proposition \ref{prop:a4inf} that 
\begin{equation*}
-\Res_{t=\infty} f_3=1, \,\,
-\Res_{t=\infty} f_4=-1, \,\,
-\Res_{t=\infty} f_0 =1, \,\,
-\Res_{t=\infty} f_1=-1, \,\,
f_2 \equiv 0.
\end{equation*}
Furthermore, from Table 1, 
it follows that  
\begin{equation*}
(f_0,f_1), \,\,(f_4,f_0) \,\,
(f_1,f_3),\,\,(f_3,f_0), \,\,
(f_3,f_4,f_0,f_1)
\end{equation*}
can have a pole in $\C.$ 
Since $f_4 \not\equiv 0$ and $\Res_{t=\infty} f_4 \neq 0,$ 
$f_4$ has a pole at $t=c \in \C$. 
If $(f_4,f_0)$ or $(f_3,f_4,f_0,f_1)$ has a pole at $t=c \in \C,$ 
it follows 
from Proposition \ref{prop:a4c} that 
$\Res_{t=c} f_4=1, \text{or} \,\,3,$ which contradicts the residue theorem. 
\par
When $f_4,f_0,f_2$ all have a pole at $t=\infty,$ 
it follows 
from Proposition \ref{prop:a4inf} that 
\begin{equation*}
-\Res_{t=\infty} f_4=1, \,\,
-\Res_{t=\infty} f_0=0, \,\,
f_1 \equiv 0, \,\,
-\Res_{t=\infty} f_2=-1, \,\,
f_3 \equiv 0.
\end{equation*}
Since $f_1=f_3 \equiv 0,$ 
it follows from Table 1 that 
\begin{equation*}
(f_4,f_0) \,\,
(f_0,f_2), \,\,(f_2,f_4)
\end{equation*}
can have a pole in $\C.$ 
When $(f_4,f_0)$ or $(f_0,f_2)$ has a pole at $t=c \in \C,$ 
Proposition \ref{prop:a4c} shows that 
$\Res_{t=c} f_0=-1,$ which contradicts the residue theorem. 
Therefore, $f_0$ is holomorphic in $\mathbb{C}$ and $(f_2,f_4)$ has a pole at $t=0$ 
because $\Res_{t=\infty} f_2$ and $\Res_{t=\infty} f_4$ are both odd integers. 
Proposition \ref{prop:a4c} and Corollary \ref{coro:res} imply that 
\begin{equation*}
f_4=t+\frac{1}{t}, \,\,
f_0=t, \,\,
f_1 \equiv 0, \,\,
f_2=-t-\frac{1}{t}, \,\,
f_3 \equiv 0,
\end{equation*}
By substituting this solution into $A^{(1)}_4(1,0,0,0,0)$, 
we can then show the contradiction.
\end{proof}

\subsection{Rational solutions of Type B for some parameters}

\begin{table}[h]
\caption
{
the residues of $ \hat{H} $ at $ t=c \in \C$ 
for $A_4^{(1)}( 1/3  ,  1/3 ,  1/3 , 0,0 ) $
}
\begin{center}
\begin{tabular}{|c|c|c|c|c|c|}
\hline
$i$ & 0 & 1 & 2 & 3 & 4  \\
\hline
$ \alpha_{ i + 2 } + \alpha_{ i + 4 } $ & $ 1/3  $ & $ 1/3  $ & 
$  1/3  $ & $ 2/3  $ & $ 1/3  $ \\
\hline
$ \alpha_{ i + 1 } $ & $ 1/3  $ & $ 1/3 $ & $ 0$ & 0 & $1/3$ \\
\hline
$ \alpha_{ i + 1 } + \alpha_{ i + 4 } $ & $ 1/3  $ & $ 2/3  $ & 
$ 1/3  $ & $ 1/3  $ & $ 1/3  $ \\
\hline
\end{tabular}
\end{center}
\end{table}

Using Table 2, 
we determine the rational solution of Type B of 
$A^{(1)}_4(1/3,1/3,1/3,0,0).$

\begin{proposition}
\label{prop:Bpoint1}
{\it
$A^{(1)}_4(1/3,1/3,1/3,0,0)$ 
has a rational solution of Type B 
such that
$$
(f_0,f_1,f_2,f_3,f_4)
=
(
t/3,
t/3,
t/3,
0,
0
),
$$
and it is unique. 
}
\end{proposition}

\begin{proof}
By direct calculation, 
we find that $A^{(1)}_4(1/3,1/3,1/3,0,0)$ 
has a rational solution of Type B 
such that 
$
(f_0,f_1,f_2,f_3,f_4)
=(t/3,t/3,t/3,0,0). 
$ 
Furthermore, 
it follows from Proposition \ref{prop:a4inf} 
that 
if $f_0,f_1,f_2$ all have a pole at $t=\infty$ and 
$f_3,f_4$ are both holomorphic at $t=\infty,$ 
the solution is unique. 
\par
We also consider the other cases: 
for some $i=1,2,3,4,$ 
$
f_i,f_{i+1},f_{i+2}
$
all have a pole at $t=\infty$ and 
$f_{i+3}, f_{i+4}$ are both holomorphic at $t=\infty.$  
\par
If $f_1,f_2,f_3$ all have a pole at $t=\infty,$ 
it follows 
from Proposition \ref{prop:a4inf} that 
$$
-\Res_{t=\infty} f_1=-\Res_{t=\infty} f_2 =0,
-\Res_{t=\infty} f_3=1, f_4 \equiv 0, -\Res_{t=\infty} f_0=-1.
$$
Proposition \ref{prop:phinf} shows that 
$h_{\infty,-1}=-\mathrm{Res}_{t=\infty} \hat{H}=0.$ 
Furthermore 
Proposition \ref{prop:H} and Table 2 imply that 
$\Res_{t=c} \hat{H}$ is nonnegative 
when 
$
(\alpha_0,\alpha_1,\alpha_2,\alpha_3,\alpha_4) 
=
(1/3, 1/3, 1/3,0,0).
$
Thus, the residue theorem shows that $\Res_{t=c} \hat{H}=0$. 
Therefore, Proposition \ref{prop:H} and Table 2 imply that 
only $(f_2,f_4)$ and $(f_3,f_0)$ can have a pole at $t=c \in \mathbb{C}$. 
Since $f_4 \equiv 0,$ $f_4$ cannot have a pole in $\mathbb{C}$. 
If $(f_3,f_0)$ has a pole at $t=c \in \mathbb{C},$ 
it follows 
from Proposition \ref{prop:a4c} that 
$\Res_{t=c} f_3=-1$ and $\Res_{t=c} f_0=1,$ 
which contradicts the residue theorem.
\par
If $f_2,f_3,f_4$ all have a pole at $t=\infty$, 
it follows 
from Proposition \ref{prop:phinf} that $h_{\infty,-1}=-4/9,$ 
which contradicts Proposition \ref{prop:fund}.
\par
If $f_3,f_4,f_0$ have a pole at $t=\infty$, 
it follows 
from Proposition \ref{prop:phinf} that $h_{\infty,-1}=-10/27,$ 
which contradicts Proposition \ref{prop:fund}.
\par
If $f_4,f_0,f_1$ have a pole at $t=\infty,$ 
it follows 
from Proposition \ref{prop:a4inf} that 
$$
-\Res_{t=\infty} f_4=-1
-\Res_{t=\infty} f_0=0,
-\Res_{t=\infty} f_1=0, 
-\Res_{t=\infty} f_2=1, 
f_3 \equiv 0.
$$
Proposition \ref{prop:phinf} implies that 
$h_{\infty,-1}=-\mathrm{Res}_{t=_infty} \hat{H}=0.$ 
Table 1 shows that 
$\mathrm{Res}_{t=c} \hat{H}$ is nonnegative 
when 
$
(\alpha_0,\alpha_1,\alpha_2,\alpha_3,\alpha_4) 
=
(1/3, 1/3, 1/3,0,0).
$ 
Thus, it follows from the residue theorem that $\Res_{t=c} \hat{H}=0$ for any $c \in \mathbb{C}.$ 
Therefore, Proposition \ref{prop:H} and Table 2 imply that 
only $(f_2,f_4)$ and $(f_3,f_0)$ can have a pole in $\mathbb{C}$. 
Since $f_3 \equiv 0,$ $f_3$ cannot have a pole in $\mathbb{C}$. 
If $(f_2,f_4)$ has a pole at $t=c \in \mathbb{C}$, 
it follows 
from Proposition \ref{prop:a4c} that 
$\mathrm{Res}_{t=c} f_2=-1$ and $\Res_{t=c} f_4=1,$ 
which contradicts the residue theorem.
\end{proof}

Using Proposition \ref{prop:fund}, 
we prove the following proposition:

\begin{proposition}
\label{prop:Bpoint2}
{\it
$
A^{(1)}_4(2/3, 0, 0, 1/3, 0), 
$ 
$A^{(1)}_4(1/3, 0, 0, 2/3, 0),$ 
$A^{(1)}_4(0, 1/3, 0, 1/3, 1/3),$ 
$A^{(1)}_4(1,0,0,0,0)$ 
have no rational solution of Type B.
}
\end{proposition}

\begin{proof}
If the equations in the proposition have a rational solution of Type B, 
it follows 
from Proposition \ref{prop:phinf} that 
$h_{\infty,-1} <0$, 
which contradicts Proposition \ref{prop:fund}.
\end{proof}

\subsection{Rational solutions of Type C for some parameters}

\begin{proposition}
\label{prop:Cpoint1}
{\it
$A^{(1)}_4(1/5,1/5,1/5,1/5,1/5)$ 
has a rational solution of Type C such that 
$$
(f_0,f_1,f_2,f_3,f_4)
=
(
t/5,
t/5,
t/5,
t/5,
t/5
),
$$
and it is unique. 
}
\end{proposition}

\begin{proof}
The proposition follows from Proposition \ref{prop:a4inf}. 

\end{proof}

By using Proposition \ref{prop:fund}, 
we prove the following proposition:

\begin{proposition}
\label{prop:Cpoint2}
{\it
$A^{(1)}_4(1, 0, 0, 0, 0),$
$A^{(1)}_4(3/5, 0, 1/5, 1/5, 0),$
$A^{(1)}_4( 1/5, 0, 2/5, 2/5, 0),$
$A^{(1)}_4(1/5, 2/5, 0, 0, 2/5),$
$A^{(1)}_4(3/5, 1/5, 0, 0, 1/5)$
have no rational solution of Type C.
}
\end{proposition}

\begin{proof}
If the equations of the proposition have a rational solution 
of Type C, 
it follows 
from Proposition \ref{prop:phinf} that 
$h_{\infty,-1} <0,$ 
which contradicts Proposition \ref{prop:fund}.
\end{proof}

\section{Proof of Main Theorem}

\subsection{Main theorem for Type A}

\begin{proof}

Theorem \ref{thm:a4par-A} 
proves that 
if $A^{(1)}_4(\alpha_i)_{0 \leq i \leq 4}$ has a rational solution of Type A, 
the parameters $\alpha_i \,\,(0 \leq i \leq 4)$ are all integers. 
Proposition \ref{prop:reduction-typeA} 
shows that $(\alpha_i)_{0\leq i \leq 4}$ 
can be transformed into 
$(1,0,0,0,0).$ 
Proposition \ref{prop:Apoint} 
proves that 
$A^{(1)}_4(1,0,0,0,0)$ 
has a rational solution of Type A
such that 
\begin{equation*}
(f_0,f_1,f_2,f_3,f_4)
=
(t,0,0,0,0).
\end{equation*}
and it is unique. 
Therefore, 
$A^{(1)}_4(\alpha_i)_{0\leq i \leq 4}$ has a rational solution of Type A 
if and only if $\alpha_i \,\,(0\leq i \leq 4)$ are all integers. 
Furthermore, the rational solution is unique 
and can be transformed into 
$$
(f_0 ,f_1, f_2, f_3, f_4)=(t,0,0,0,0)
 \,\, \mathrm{with} \,\,
(\alpha_0, \alpha_1, \alpha_2, \alpha_3, \alpha_4) =
(1,0,0,0,0). 
$$

\end{proof}

\subsection{Main theorem for Type B}

\begin{proof}

Theorem \ref{thm:a4par-B} 
implies that 
if $A^{(1)}_4(\alpha_i)_{0\leq i \leq 4}$ has a rational solution of Type B 
for some $i=0,1,2,3,4,$ 
$$ 
(\alpha_i, \alpha_{i+1}, \alpha_{i+2}, \alpha_{i+3}, \alpha_{i+4})
\equiv
\left( 
\frac{n_1}{3}-\frac{n_3}{3}, \, \frac{n_1}{3}, \,  
\frac{n_1}{3}+\frac{n_4}{3}, \, \frac{n_3}{3}, \,  
- \frac{n_4}{3}
\right) \quad \textrm{mod} \, 
\mathbb{Z}, 
$$
where $n_1, n_3, n_4=0, 1, 2.$ 
Proposition \ref{prop:reduction-typeB}
shows that 
the parameters can be transformed so that 
\begin{align*} 
(\alpha_0,\alpha_1,\alpha_2,\alpha_3,\alpha_4)=&
(1/3, 1/3, 1/3, 0, 0), 
(2/3, 0, 0, 1/3, 0), 
(1/3, 0, 0, 2/3, 0), \\
&(0, 1/3, 0, 1/3, 1/3), 
(1,0,0,0,0),
\end{align*}
and 
that 
the parameters are transformed into 
$(1/3, 1/3, 1/3, 0, 0)$  
if 
for some $ i = 0, 1, \ldots 4,$
$$
\quad  
(\alpha_i, \alpha_{i+1}, \alpha_{i+2}, \alpha_{i+3}, \alpha_{i+4}) \equiv
\left\{
\begin{matrix}
\pm \frac13(1, 1, 1, 0, 0) \quad {\rm mod}   \, \Z \\
\pm \frac13(1, -1, -1, 1, 0)\quad {\rm mod}   \, \Z. 
\end{matrix}
\right.
$$
Otherwise, 
the parameters can be transformed so that 
\begin{align*} 
(\alpha_0,\alpha_1,\alpha_2,\alpha_3,\alpha_4)=&
(2/3, 0, 0, 1/3, 0), 
(1/3, 0, 0, 2/3, 0), \\
&(0, 1/3, 0, 1/3, 1/3),
(1,0,0,0,0).
\end{align*}
Proposition \ref{prop:Bpoint1} 
shows that 
$A^{(1)}_4(1/3, 1/3, 1/3, 0, 0)$ has a unique rational solution 
which is given by 
\begin{equation*}
(f_0,f_1,f_2,f_3,f_4)
=
(
\frac{t}{3},
\frac{t}{3},
\frac{t}{3},
0,0
).
\end{equation*}
Furthermore, from Proposition \ref{prop:Bpoint2}, 
it follows that 
none of 
$A^{(1)}_4(2/3, 0, 0, 1/3, 0), $ 
$A^{(1)}_4(1/3, 0, 0, 2/3, 0), $ 
$A^{(1)}_4(0, 1/3, 0, 1/3, 1/3), $ 
and 
$A^{(1)}_4(1,0,0,0,0)$ 
has a rational solution of Type B. 
\par
Therefore, 
$A^{(1)}_4(\alpha_i)_{0\leq i \leq 4}$ 
has a rational solution of Type B 
if and only if 
for some $ i = 0, 1, \ldots 4$,
$$
(\alpha_i, \alpha_{i+1}, \alpha_{i+2}, \alpha_{i+3}, \alpha_{i+4}) 
\equiv
\left\{
\begin{matrix}
\pm 1/3(1, 1, 1, 0, 0) \quad {\rm mod}   \, \mathbb{Z} \\
\pm 1/3(1, -1, -1, 1, 0)\quad {\rm mod}   \, \mathbb{Z}. 
\end{matrix}
\right.
$$
Moreover,
the rational solution is unique 
and 
can be transformed into 
$$
(f_0 ,f_1, f_2, f_3, f_4) =
(t/3, t/3, t/3, 0, 0) \, 
\mathrm{with} 
\,
(\alpha_0, \alpha_1, \alpha_2, \alpha_3, \alpha_4)= 
(1/3, 1/3, 1/3,0,0).
$$

\end{proof}

\subsection{Main theorem for Type C}

\begin{proof}

Theorem \ref{thm:a4par-C}
proves that 
if $A^{(1)}_4(\alpha_k)_{0\leq k \leq 4}$ has a rational solution of Type C 
for some $i=0,1,2,3,4,$ 
$$
(\alpha_i, \alpha_{i+1}, \alpha_{i+2}, \alpha_{i+3}, \alpha_{i+4})
\equiv
\left( 
\frac{n_1}{5}+\frac{2n_2}{5}+\frac{3n_3}{5}, \, 
\frac{n_1}{5}+\frac{2n_2}{5}+\frac{n_3}{5}, \, \frac{n_1}{5}, \, 
\frac{n_1}{5}+\frac{n_2}{5}, \, \frac{n_1}{5}+\frac{n_3}{5} 
\right) 
\,
\textrm{mod} \, \Z,
$$ 
where 
$
n_1, n_2, n_3=0, 1, 2, 3, 4.
$ 
Proposition \ref{prop:reduction-typeC}
shows that 
the parameters can be transformed 
so that 
\begin{align*}
(\alpha_0,\alpha_1,\alpha_2,\alpha_3,\alpha_4)=&
(1/5, 1/5, 1/5, 1/5, 1/5), 
(1, 0, 0, 0, 0), (3/5, 0, 1/5 1/5, 0), 
(1/5, 0, 2/5, 2/5, 0), \\
&(1/5, 2/5, 0, 0, 2/5),
(3/5, 1/5, 0, 0, 1/5) 
\end{align*}
and that 
the parameters are 
transformed into 
$(1/5, 1/5, 1/5, 1/5, 1/5)$ 
if 
for some $i = 0, 1, \ldots, 4,$
$$
(\alpha_i, \alpha_{i+1}, \alpha_{i+2}, \alpha_{i+3}, \alpha_{i+4}) \equiv
\left\{
\begin{matrix}
  j/5(1, 1, 1, 1, 1) \quad {\rm mod} \, \Z    \\
  j/5(1, 2, 1, 3, 3) \quad {\rm mod} \, \Z,   
\end{matrix}
\right.
$$
with some $ j=1, 2, 3, 4.$ 
Otherwise, the parameters are transformed 
\begin{align*}
(\alpha_0,\alpha_1,\alpha_2,\alpha_3,\alpha_4)=&
(1, 0, 0, 0, 0), (3/5, 0, 1/5 1/5, 0), 
(1/5, 0, 2/5, 2/5, 0), \\
&(1/5, 2/5, 0, 0, 2/5),
(3/5, 1/5, 0, 0, 1/5). 
\end{align*}
\par
Proposition \ref{prop:Cpoint1} 
implies that 
$A^{(1)}_4(1/5, 1/5, 1/5, 1/5, 1/5)$ 
has a unique rational solution of Type C 
such that 
$$
(f_0,f_1,f_2,f_3,f_4)
=
(
t/5,
t/5,
t/5,
t/5,
t/5
).
$$
From Proposition \ref{prop:Cpoint2}, 
it follows that 
none of 
$A^{(1)}_4(1, 0, 0, 0, 0),$
$A^{(1)}_4(3/5, 0, 1/5, 1/5, 0),$
$A^{(1)}_4(1/5, 0, 2/5, 2/5, 0),$
$A^{(1)}_4(1/5, 2/5, 0, 0, 2/5)$ 
and 
$A^{(1)}_4(3/5, 1/5, 0, 0, 1/5)$
has a rational solution of Type C. 
\par
Therefore, 
$A^{(1)}_4(\alpha_j)_{0\leq j \leq 4}$ 
has a rational solution of Type C 
if and only if 
for some $i = 0, 1, \ldots, 4,$
$$
(\alpha_i, \alpha_{i+1}, \alpha_{i+2}, \alpha_{i+3}, \alpha_{i+4}) \equiv
\left\{
\begin{matrix}
  j/5(1, 1, 1, 1, 1) \quad {\rm mod} \, \Z    \\
  j/5(1, 2, 1, 3, 3) \quad {\rm mod} \, \Z,   
\end{matrix}
\right.
$$
with some $ j=1, 2, 3, 4.$ 
Furthermore, 
the rational solution is unique 
and 
can be transformed into 
$$
(f_0 ,f_1, f_2, f_3, f_4) =
(t/5, t/5, t/5, t/5, t/5) \, 
\mathrm{with} 
\,
(\alpha_0, \alpha_1, \alpha_2, \alpha_3, \alpha_4)= 
(1/5, 1/5, 1/5, 1/5, 1/5). 
$$

\end{proof}
\par
We complete the proof of the main theorem.


\begin{thebibliography}{5}
\bibitem{Adler}
Adler, V. E., 
Nonlinear chains and Painlev\'e equations, 
Physica, {\bf D 73}, (1994), 335-351.
\bibitem{Clarkson}
Clarkson. P. A., 
New exact solutions of the Boussinesq equation. 
European J. Appl. Math. {\bf 1} (1990), no. 3, 279-300.
\bibitem{FlaschkaNewell}
Flaschka. N. P. and Newell. A. C., 
Monodromy- and spectrum-preserving deformations. I. 
Comm. Math. Phys. {\bf 76} (1980), no. 1, 65-116.
\bibitem{Gambier}
Gambier, B., 
Sur les \'equations diff\'erentielles du second ordre et du premier degre dont  
l'int\'egrale est \'a points critiques fixes, 
Acta Math, {\bf 33}, (1909), 1-55.
\bibitem{Garnier}
Garnier, R., 
Sur des equations differentielles du troisieme ordre dontl'integrale generale est uniforme 
et sur une classe d'equationsnouvelles d'ordre superieur 
dont l'integrale generale a sespoints critiques fixes. 
Ann. Ecol. Norm. Sup. Ser. 3, {\bf 29} (1912), 1-126.
\bibitem{Gro}
Gromak, V.~I., 
Algebraic solutions of the third Painlev\'e equation (Russian),
Dokl. Akad. Nauk BSSR., {\bf 23}, (1979), 499-502.
\bibitem{Gr:83} 
Gromak, V.~I., 
Reducibility of the Painlev\'e equations, 
Differential Equations {\bf 20}, (1983), 1191-1198.
\bibitem{ItsNovokshenov}
Its, A. R. and Novokshenov, Y. Yu., 
The Isomonodromic Deformation Method in the Theory 
of Painlev\'e Equations. 
Lecture Notes in Math., 
1191, 
Springer-Verlag, Berlin, 1986.
\bibitem{Kit-Law-McL}
Kitaev, A.~V., Law, C.~K. and McLeod, J.~B., 
Rational solutions of the fifth Painlev\'e equation. 
Differential Integral Equations, {\bf 7}, (1994), 967-1000.
\bibitem{LeauteMarcilhacy1}
Leaute, B. and Marcilhacy, G., 
A new transcendent solution of Einstein's equations. 
Phys. Lett. A {\bf 87} (1981/82), no. 4, 159-161.
\bibitem{LeauteMarcilhacy2}
Leaute, B. and Marcilhacy, G., 
On a particular transcendent solution of the Ernst system generalized on $n$ fields. 
J. Math. Phys. {\bf 27} (1986), no. 3, 703-706.
\bibitem{LeauteMarcilhacy3}
Leaute, B. and Marcilhacy, G., 
A class of self-dual solutions for ${\rm SU}(2)$ gauge fields on Euclidean space. 
J. Math. Phys. {\bf 28} (1987), no. 4, 774-776. 
\bibitem{Matsu}
Matsuda, K., 
Rational solutions of the $A_4$ Painlev\'e equation.
Proc. Japan Acad. Ser. A. {\bf81}, No. 5, (2005), 85-88. 
\bibitem{Mazzo}
Mazzocco, M., 
Rational Solutions of the Painlev\'e VI Equation.
Kowalevski Workshop on Mathematical Methods of Regular Dynamics (Leeds, 2000). 
J. Phys. A., {\bf 34}, (2001), 2281-2294.
\bibitem{MorrisDodd}
Morris, H. C. and Dodd, R. K., 
Geometric structures and inverse scattering problems for Painleve transcendents. 
Phys. Lett. A {\bf 75} (1979/80), no. 4, 249--253.
\bibitem{Mura1}
Murata, Y., 
Rational solutions of the second and the fourth Painlev\'e equations. 
Funkcial. Ekvac., {\bf 28}, (1985), 1-32. 
\bibitem{NoumiYamada-A}
Noumi, M. and Yamada, Y., 
Affine Weyl Groups, Discrete Dynamical Systems and Painlev\'e Equations.
Comm. Math. Phys., {\bf 199}, (1998), 281-295.
\bibitem{NoumiYamada-B}
Noumi, M. and Yamada, Y., 
Higher order Painlev\'e equations of type $A^{(1)}_l$, 
Funkcial. Ekvac. {\bf41}, (1998), 483-503. 
\bibitem{oka1}
Okamoto, K., 
Studies on the Painleve equations. III. 
Second and fourth Painleve equations, $P\sb {{\rm II}}$ and $P\sb {{\rm IV}}$, 
Math. Ann. {\bf 275}, no. 2, (1986), 221-255.
\bibitem{oka2}
Okamoto, K., 
Studies on the Painlev\'e equations. I. Sixth Painleve equation $P\sb {{\rm VI}}$, 
Ann. Mat. Pura Appl. (4) {\bf 146}, (1987), 337-381. 
\bibitem{oka3}
Okamoto, Kazuo., 
Studies on the Painleve equations. II. Fifth Painleve equation $P\sb {\rm V}$, 
Japan. J. Math. (N.S.) {\bf 13}, no. 1, (1987), 47-76. 
\bibitem{oka4}
Okamoto, K., 
Studies on the Painleve equations. IV. Third Painleve equation $P\sb {{\rm III}}$,  
Funkcial. Ekvac. {\bf 30}, no. 2-3, (1987), 305-332.
\bibitem{Painleve}
Painlev\'e, P., 
Sur les \'equations diff\'erentielles du second ordre et d\'ordre sup\'erieur dont 
l'int\'egrale g\'en\'erale est uniforme, Acta Math. {\bf 25}, (1902), 1-85.
\bibitem{Salihoglu}
Salihoglu, S., 
The two-dimensional ${\rm O}(N)$ nonlinear $\sigma $-model and the fifth Painleve transcendent. 
Phys. Lett. B {\bf 89} (1980), no. 3-4, 367-368.
\bibitem{Tahara}
Tahara, N., 
An augmentation of the phase space of the system 
of type $A\sb 4\sp {(1)}$, 
Kyushu J. Math. {\bf 58}, no. 2, (2004), 393-425.
\bibitem{VeselovShabat}
Veselov. A. P. 
and Shabat. A. B., 
A dressing chain and the spectral theory of the Schlesinger operator, 
Funct. Anal. Appl {\bf 27}, (1993), 81-96.
\bibitem{Vorob}
Vorob'ev, A. P., 
On the rational solutions of the second Painlev\'e equation, 
Differential Equations, {\bf 1}, (1965), 58-59.
\bibitem{WuMcCoyTracyBarouch}
Wu. T. T., 
McCoy. B. M., 
Tracy. C. A 
and Barouch. E., 
Spin-spin correlation functions for the two-dimensional Ising model, 
exact theory in the scaling region, 
Phys. Rev., 
B {\bf 13}, (1976), 316-374.
\bibitem{Yab:59} 
Yablonskii A. I., 
On rational solutions of the second Painlev\'e equation  (Russian), 
Vesti. A. N. BSSR, Ser. Fiz--Tekh. Nauk., {\bf 3}, (1959), 30-35. 
\bibitem{YuangLi}
Yuang Wenjun and Li Yezhou, 
Rational Solutions of Painlev\'e Equations, 
Canad. J. Math. 
Vol. {\bf 54} (3), 
(2002), 
648-670.
\end{thebibliography}
\end{document}